\documentclass[12pt]{article}

\usepackage{upgreek}
\usepackage{listings}
\usepackage{amsmath,amssymb}
\usepackage{amsthm}
\usepackage{multirow}
\usepackage{makecell}
\usepackage{mathrsfs}
\usepackage{pgf,tikz}
\usepackage{pgfplots}
\usepackage{graphicx}
\usepackage{caption}
\usepackage[ruled,vlined,linesnumbered]{algorithm2e}
\usepackage{bm}
\usepackage{bbm}
\usepackage{enumitem}
\usepackage{lscape}
\usetikzlibrary{arrows}
\usepackage{geometry}
\usepackage{booktabs}
\usepackage[utf8]{inputenc}
\usepackage{url} 
\usepackage{hyperref}

\renewcommand\leq{\leqslant}
\renewcommand\geq{\geqslant}
\usepackage[normalem]{ulem} 

\newtheorem{theorem}{Theorem}[section] 
\newtheorem{proposition}[theorem]{Proposition}

\theoremstyle{definition}
\newtheorem{definition}[theorem]{Definition}
\newtheorem{lemma}[theorem]{Lemma}
\newtheorem{corollary}[theorem]{Corollary}

\theoremstyle{definition}

\theoremstyle{remark}
\newtheorem{remark}{Remark}

\theoremstyle{definition}
\newtheorem{example}{Example}[section]

\DeclareMathOperator{\conv}{conv}
\DeclareMathOperator{\aff}{aff}

\DeclareMathOperator{\supp}{supp}
\DeclareMathOperator{\proj}{Proj}

\usepackage{ragged2e}
\usepackage{array}
\newcolumntype{L}[1]{>{\RaggedRight\arraybackslash}p{#1}} 
\newcolumntype{C}[1]{>{\Centering\arraybackslash}p{#1}} 
\newcolumntype{R}[1]{>{\RaggedLeft\arraybackslash}p{#1}} 

\newcolumntype{N}[1]{>{\RaggedLeft\arraybackslash $ }p{#1}<{$}} 
\newcolumntype{M}{>{$}r<{$}} 
\newcolumntype{P}{>{$}l<{$}} 

\usepackage{embrac} 
\usepackage{empheq} 

\newcommand{\MC}{\mathrm{MC}}
\newcommand{\MP}{\mathrm{MP}}

\makeatletter
\@addtoreset{equation}{section}
\makeatother

\newcommand{\add}[1]{\textcolor{red}{#1}}

\begin{document}
\title{The Boolean polynomial polytope with multiple choice constraints}
\author{Sihong Shao\footnotemark[1]
\and Yishan Wu\footnotemark[2]}

\renewcommand{\thefootnote}{\fnsymbol{footnote}}
\footnotetext[1]{CAPT, LMAM and School of Mathematical Sciences, Peking University, Beijing 100871, China. Email: \texttt{sihong@math.pku.edu.cn}}
\footnotetext[2]{CAPT, LMAM and School of Mathematical Sciences, Peking University, Beijing 100871, China. Email: \texttt{wuyishan@pku.edu.cn}}

\maketitle



\begin{abstract}
We consider a class of $0$-$1$ polynomial programming termed multiple choice polynomial programming (MCPP)
where the constraint requires exact one component per subset of the partition to be $1$ after all the entries are partitioned. 
Compared to the unconstrained counterpart, there are few polyhedral studies of MCPP in general form. This paper serves as the first attempt to propose a polytope associated with a hypergraph to study MCPP, which is the convex hull of $0$-$1$ vectors satisfying multiple choice constraints and production constraints. With the help of the decomposability property, we obtain an explicit half-space representation of the MCPP polytope when the underlying hypergraph is $\alpha$-acyclic by induction on the number of hyperedges, which is an analogy of the acyclicity results on the multilinear polytope by Del Pia and Khajavirad (SIAM J Optim 28 (2018) 1049) when the hypergraph is $\gamma$-acyclic.
We also present a necessary and sufficient condition for the inequalities lifted from the facet-inducing ones for the multilinear polytope to be still facet-inducing for the MCPP polytope. This result covers the particular cases by B\"armann, Martin and  Schneider (SIAM J Optim 33 (2023) 2909).



\vspace*{4mm} 

\end{abstract}

\section{Introduction}
\label{sec:introduction}
We consider a constrained $0$-$1$ polynomial programming problem of the following form:
\begin{subequations}
\label{eq:MCPP_pseudo_Bool_def}
\begin{align}
\max_{x\in \{0,1\}^n} \quad & f(x)=\sum_{J\in \mathcal{J}} a_J\prod_{i\in J} x_i,\label{eq:multilinear_form}\\
\text{s.t.} \quad & \sum_{i\in I} x_i=1,\ \forall\,I\in \mathscr{P},\label{eq:sum_x_i_1}
\end{align}
\end{subequations}
where $\mathcal{J}\subseteq 2^{[n]}$ is a family of nonempty subset of the index set $[n]:=\{1,2,\dots,n\}$, $a\in \mathbb{R}^{\mathcal{J}}$ is the vector of coefficients, and $\mathscr{P}$ is a partition of $[n]$, i.e., the disjoint union of all $I\in \mathscr{P}$ is $[n]$. Without loss of generality, we assume that $|I|>1$ for all $I\in \mathscr{P}$. 
The constraints in Eq.~\eqref{eq:sum_x_i_1} imply that there is exactly one entry of $x$ per subset of the partition equal to $1$. Thus the entries in each $I\in\mathscr{P}$ serve as a decision with $|I|$ choices. Therefore, the constraints are often termed \textit{multiple choice constraints} \cite{nauss_0-1_1978}, and Problem~\eqref{eq:MCPP_pseudo_Bool_def} is termed \textit{multiple choice polynomial programming} (MCPP) in this paper. The feasible region of Problem~\eqref{eq:MCPP_pseudo_Bool_def} is denoted by $\mathcal{X}$. Since for all $x\in \mathcal{X}$, any monomial of $f(x)$ with the divisor $x_i x_{i'}$, $i\neq i'$, $i,i'\in I$ and $I\in \mathscr{P}$, vanishes, we can eliminate them from $f(x)$. Therefore, 
it can be further assumed that, for each $J\in \mathcal{J}$,  we have
\begin{equation}
\label{eq:J_cond}
|J\cap I|\leq 1,\quad \forall\,I\in \mathscr{P}.
\end{equation}
In this situation, we call the objective function $f(x)$ to be \emph{$\mathscr{P}$-multilinear} and
let $\mathcal{J}_{\mathscr{P}}$ collect all the nonempty subsets $J$s of $[n]$ that satisfy Eq.~\eqref{eq:J_cond}.
Thus $\mathcal{J}\subseteq\mathcal{J}_{\mathscr{P}}$.


 
Since $(x_i)^p=x_i$ for $x_i\in \{0,1\}$, any $0$-$1$ polynomial programming with multiple choice constraints can be reformulated into MCPP.  The polynomial $f(x)$ is often referred to as ``multilinear function" \cite{boros_pseudo-boolean_2002}.  
Note that if $|I|=2$ for all $I\in \mathscr{P}$, $\mathcal{X}$ is isomorphic to $\{0,1\}^{|\mathscr{P}|}$, and Problem~\eqref{eq:MCPP_pseudo_Bool_def} degenerates to a pseudo-Boolean optimization instance by substituting half of the components of $x$ according to the multiple choice constraints (i.e., assumed $I=\{i,i'\}\in \mathscr{P}$, we can substitute $x_{i'}$ by $1-x_i$). Hence, MCPP subsumes pseudo-Boolean optimization (PBO) and can be regarded as a generalization of the latter: From binary to ``$|I|$-ary''. Since PBO can model various combinatorial problems \cite{boros_pseudo-boolean_2002}, MCPP is capable of modeling problems generalized from them, for example, MAX-$k$-CUT \cite{frieze_improved_1997} (generalized from MAX-CUT) and the hypergraph k-cut \cite{chandrasekaran_hypergraph_2021} (generalized from hypergraph cut), MAX-CSP \cite{deineko_approximability_2008} (generalized from MAX-SAT). Meanwhile, there exist other problems naturally associated with multiple choice decision which can be formulated as MCPP, such as the calculation of star discrepancy \cite{shao_ode_2023}. It is worth mentioning that all these  problems are NP-hard.

This work is devoted into performing a polyhedral study of MCPP. To this end, we need to first introduce a hypergraph $H(V, E)$ corresponding to Problem~\eqref{eq:MCPP_pseudo_Bool_def} with 
\begin{equation}
V := \mathscr{P}, \quad E := \{E(J)|J\in \mathcal{J}, |E(J)|>1\},
\end{equation}
where $E(J):=\{I\in\mathscr{P}||I\cap J|=1\}$ is defined for each $J \in \mathcal{J}_{\mathscr{P}}$. To facilitate the discussion, the family $\mathcal{J}$ is enlarged into $\mathcal{J}^H$ by adding monomials for ensuring the subset-uniform property 
with respect to (w.r.t) $V$ 
and $\{i\}\in \mathcal{J}^H$ for any $i\in [n]$.  
Actually, it can be shown that for any $\mathcal{J}$, there exists $\mathcal{J}^H$ such that $\mathcal{J}\subseteq \mathcal{J}^H$ and thus after possibly introducing zero entries of the coefficient vector $a$ for the added monomials in $\mathcal{J}^H\setminus \mathcal{J}$, we can replace $\mathcal{J}$ with $\mathcal{J}^H$ in Eq.~\eqref{eq:multilinear_form}. More explanations on this point and the family $\mathcal{J}^H$ are left for Section~\ref{subsec:intro_J_H}. After representing each monomial $\prod_{i\in J} x_i$ with a new variable $w_J$ for $J\in \mathcal{J}^H$ --- a common practice to obtain a linear objective function from a polynomial one, and adding the production constraints $w_J=\prod_{i\in J} w_i$,
Problem~\eqref{eq:MCPP_pseudo_Bool_def} with $\mathcal{J}^H$ equivalently becomes 
\begin{subequations}
\label{eq:Linearized_MCPP}
\begin{align}
\max_{w\in \{0,1\}^{\mathcal{J}^H}}\quad & \sum_{J\in \mathcal{J}^H} a_Jw_J,\\
\text{s.t.} \quad & w(I) = 1,\ \forall\, I\in \mathscr{P},\label{eq:w_J_multichoice_cons}\\
& w_J=\prod_{i\in J} w_i,\ \forall\, J\in \mathcal{J}^H,\ |J|>1, \label{eq:w_J_prod_cons}
\end{align}
\end{subequations}
where we have applied the conventions: $w_i: = w_{\{i\}}=x_i$ for all $i\in [n]$,
and $w(A)=\sum_{i\in A} w_{i}$ for any nonempty subset  $A\subseteq [n]$. 
These conventions also apply to other vectors than $w$ hereafter. 
In order to develop the polytope theory for Problem~\eqref{eq:Linearized_MCPP}, 
it is key to understand the structure of the convex hull of its feasible region. 
We refer to the feasible region as \emph{$\mathscr{P}$-multilinear set} w.r.t. $H(V, E)$, denoted by $\mathscr{S}^H$, i.e.,
\begin{equation}
\label{eq:def_SG_MCMP_set}
\mathscr{S}^H=\left\{w\in \{0,1\}^{\mathcal{J}^H} \middle|w \text{ satisfies Eqs.~\eqref{eq:w_J_multichoice_cons} and \eqref{eq:w_J_prod_cons}}\right\}.
\end{equation}
The subsequent discussion focuses on $\MC^H$ --- the convex hull of $\mathscr{S}^H$: $\MC^H =\conv  \mathscr{S}^H$,  and we call it the Boolean polynomial polytope with multiple choice constraints or the MCPP polytope. In particular, the structural properties of $\MC^H$, including its half-space representation and lifting facets from related polytopes,  are concerned.


\subsection{Contributions of this work}

The first contribution gives an explicit half-space representation of $\MC^H$ under the condition that $H$ is $\alpha$-acyclic.
The $\alpha$-acyclic hypergraph belongs to the most general case among the well-known four degrees of acyclicity: 
$\alpha$, $\beta$, $\gamma$ and Berge-acyclic (see Theorem 6.1 of \cite{fagin_degrees_1983}). 
An important characterization is that $\alpha$-acyclic hypergraphs have join trees \cite{beeri_desirability_1983}.

\begin{theorem}[Half-space Representation]
\label{thm:alpha_acyclic}
Suppose $H(V,E)$ is $\alpha$-acyclic and $T=(E,\mathcal{E})$ is a join tree of $H$. Then
\begin{subequations}
\label{eq:MC_H_repres_alpha_acyclic}
\begin{align}
\MC^H=\bigg\{w\in \mathbb{R}_{\geq 0}^{\mathcal{J}^H}\bigg|
&w(I)=1,\ \forall\, I\in V;\\
&w_i - \sum_{J\in \mathcal{J}^e: J\ni i}w_J = 0,\ \forall\, e\in E, i\in I\in e;\\
&\sum_{J\in \mathcal{J}^e: J\supseteq J_0} w_J -\sum_{J\in \mathcal{J}^{e'}: J\supseteq J_0}w_J = 0,\ \forall\, \{e,e'\}\in \mathcal{E},|e\cap e'|>1,J_0\in \mathcal{J}^{e\cap e'}\bigg\}.
\end{align}
\end{subequations}
\end{theorem}

Theorem~\ref{thm:alpha_acyclic} allows us to solve Eq.~\eqref{eq:MCPP_pseudo_Bool_def} in polynomial time when the rank of $H$ and $|I|$ for $I\in V$ are bounded.

We prove the above Theorem~\ref{thm:alpha_acyclic} by induction on $|E|$. The inductive step mainly consists of deducting description of $\MC^H$ from known descriptions of $\MC^{H_1}$ and $\MC^{H_2}$, where $H_1$ and $H_2$ are subhypergraphs of $H$. The crucial part of the deduction is a sufficient condition
for decomposability of $\MC^H$. We present it in Theorem~\ref{thm:decomp} of independent interest.

The second contribution is a verifiable necessary and sufficient condition for a lifted inequality to be facet-inducing. 
The so-called ``multilinear polytope'', $\MP^H:=\conv\mathcal{S}^H$, was introduced to study PBO \cite{del_pia_polyhedral_2017}
\begin{equation}
\label{eq:PBO}
\max_{x\in \{0,1\}^V} \sum_{I\in V}p_Ix_I+\sum_{e\in E}p_e\prod_{I\in e}x_I,
\end{equation}
where $\mathcal{S}^H$ is the multilinear set
\begin{equation}
\label{eq:multilinear_set}
\mathcal{S}^H=\left\{z\in \{0,1\}^{L(V)\cup E}\middle| z_e=\prod_{I\in e} z_I,\ \forall\, e\in E\right\},
\end{equation}
$L(V):=\{\{I\}|I\in V\}$ denotes the collection of all singleton subsets of $V$, and $z_I: = z_{\{I\}}$. 
It can be readily seen that compared to $\mathscr{S}^H$ in Eq.~\eqref{eq:def_SG_MCMP_set} for MCPP, $\mathcal{S}^H$ for PBO lacks the multiple choice constraints. Suppose $c\cdot z\leq \delta$ is a valid inequality for $\MP^H$. For any $S_I\subseteq I,I\in V$, we have $0\leq w(S_I)\leq w(I) = 1$
for any $w\in \mathscr{S}^H$. 
Thus, we can replace $z_{I}$ with $w(S_I)$ and $z_e$ with $\prod_{I\in e} w(S_I)$, then 
linearize the inequality to obtain a valid inequality for $\MC^H$. More precisely, the following inequality is valid for $\MC^H$:
\begin{equation}
\label{eq:lifted_valid_ineq0}
\sum_{e\in L(V)\cup E} c_e \mathscr{L} \prod_{I\in e} \sum_{i\in S_I} w_i\leq \delta,
\end{equation}
where the operator $\mathscr{L}$ replaces the monomials involving $w_i, i\in [n]$, i.e., $\prod_{i\in J}w_i$, with $w_{J}$.
This can be seen as follows: for any $w\in \mathscr{S}^H$, letting $z_{I}=w(S_I)$ for $I\in V$, and $z_e=\prod_{I\in e} \sum_{i\in S_I} w_i$ for $e\in E$, we have $z\in \mathcal{S}^H$, and the LHS of the inequality Eq.~\eqref{eq:lifted_valid_ineq0} becomes $c\cdot z$, so it holds for any $w\in \MC^H$. 
We call Eq.~\eqref{eq:lifted_valid_ineq0} the lifted inequality of $c\cdot z\leq \delta$ w.r.t. $S_I, I\in V$.
Since facets are important structure\add{s} of polytopes, it is natural to ask:
\begin{quote}
\emph{
Is the lifted inequality still facet-inducing for $\MC^H$ when $c\cdot z\leq \delta$ is facet-inducing for $\MP^H$?
}
\end{quote}

Answering the above question faces some difficulties because $\MC^H$ is not full-dimensional. We resort to the full-dimensional polytope $\MC^H_\leq$ (see Eq.~\eqref{eq:MCPP_MC_leq_def}). To utilize the properties of $\MC^H_\leq$ more straightforwardly, it is required that $\MC^H$ is isomorphic to $\MC^H_\leq$, the sufficient condition of which is $H$ being downward-closed as stated in Proposition~\ref{prop:MC_leq_bijection}.
With these at our disposal, we give a positive answer in Theorem~\ref{thm:lifted_MP0} which  
provides an easily verifiable necessary and sufficient condition.



\begin{theorem}[Lifted Facet-inducing Inequalities]
\label{thm:lifted_MP0}
Suppose the hypergraph $H(V,E)$ is downward-closed, i.e., for any $e' \subseteq e, e \in E$, $|e'| > 1$, we have $e' \in E$. Let $c \cdot z \leq \delta$ be a facet-inducing inequality for $\MP^H$. Define
\begin{align}
V_0 & =\{I\in V|\forall\, z\in \MP^H, z_{I}=0\ \text{implies}\ c\cdot z=\delta\}, \label{eq:def_V_0_V_1} \\
V_1 & =\{I\in V|\forall\, z\in \MP^H, z_{I}=1\ \text{implies}\ c\cdot z=\delta\}. \label{eq:def_V_0_V_11}
\end{align}
For any $S_I\subset I, S_I\neq \varnothing, I\in V$, the inequality \eqref{eq:lifted_valid_ineq0}
is facet-inducing for $\MC^H$ if and only if the following condition holds: 
\begin{equation}
\label{eq:lifted_facet_necc_suff}
|S_I|=1,\ \forall\, I\in V_0\ \text{and}\ |S_I|=|I|-1,\ \forall\, I\in V_1.
\end{equation}
\end{theorem}

Theorem~\ref{thm:lifted_MP0} specifies all the facets of $\MC^H$ that are lifted from $\MP^H$.
Note that $S_I$ is now restricted to $\varnothing\subset S_I\subset I$ compared to the condition $S_I\subseteq I$ discussed before. This premise ensures that for any $z\in \mathcal{S}^H$, we can construct $w\in \mathscr{S}^H$ such that $w(S_I)=z_I, I\in V$. As a result, inequality~\eqref{eq:lifted_valid_ineq0} is tight (there exists $w\in \MC^H$ that makes the equality hold) and not an implicit equality (there exists $w\in \MC^H$ that makes the inequality strict).
These two properties are necessary for Eq.~\eqref{eq:lifted_valid_ineq0} being facet-inducing.
Moreover,  if $S_{I'}=\varnothing$, then $w(S_{I'})=0$ for $w\in \MC^H$, and Eq.~\eqref{eq:lifted_valid_ineq0} can be viewed as the lifted inequality of $c'\cdot z\leq \delta'$, which is obtained by removing the terms $z_{e}$ for $I'\in e, e\in L(V)\cup E$ from $c\cdot z\leq \delta$ and also valid (but not necessarily facet-inducing). Similarly, if $S_{I'}=I'$, then $w(S_{I'})=1$ for $w\in \MC^H$. Thus
$
\mathscr{L}\left(\prod_{i\in J} w_i \right)\sum_{i\in I'} w_i=w_J
$
is implicit on $\MC^H$, implying that Eq.~\eqref{eq:lifted_valid_ineq0} can be viewed as the inequality obtained by replacing $z_{e}, I\in e$ with $z_{e-I'}$ in $c\cdot z\leq \delta$ and lifting it. Therefore, Theorem~\ref{thm:lifted_MP0} can effectively determine whether lifted inequalities are facet-inducing facets even when there exists $I\in V$ such that $S_I=\varnothing$ or $S_I=I$.

\subsection{Related works}

The study of polytope theory for $0$-$1$ polynomial programming may date back to Padberg's famous Boolean quadric polytope $\mathrm{QP}^G$ \cite{padberg_boolean_1989}, which was defined as the convex hull of the feasible region of linearized unconstrained $0$-$1$ quadratic programming associated with a graph $G(V,E)$, i.e.,
\begin{equation}
\mathrm{QP}^G=\conv\{(x,y)\in \{0,1\}^{V}\times\{0,1\}^{E}|y_{ij}=x_ix_j,\ \forall\, \{i,j\}\in E\},
\end{equation}
where $y_{ij}$ is an alias of $y_e$. Various families of facets as well as lifting properties of $\mathrm{QP}^G$ were studied there under the condition that $G$ is a complete or sparse graph. On the other hand, $\mathrm{QP}^G$ is closely related to the cut polytope \cite{barahona_cut_1986}
via a bijective linear transformation between them \cite{de_simone_cut_1990}. Since then, $\mathrm{QP}^G$ and the cut polytope have been extensively studied due to their fundamental role in Boolean quadratic programming \cite{boros_cut-polytopes_1993, sherali_simultaneous_1995, yajima_polyhedral_1998}. Actually,  the multilinear polytope $\MP^H$ degenerates to $\mathrm{QP}^G$ when hypergraph $H$ reduces to graph $G$,
and a series of studies have been conducted on $\MP^H$ \cite{del_pia_multilinear_2018, del_pia_decomposability_2018, del_pia_polynomial-size_2023}. We will further show that $\MP^H$ is isomorphic to a projection of $\MC^H$ (see Proposition~\ref{prop:MP_H_one_to_one}).


%
%
%
%
%
%
%
%


However, both $\MP^H$ and $\mathrm{QP}^G$ are associated with unconstrained programming problems. Polytopes associated with $0$-$1$ polynomial programming with multiple choice constraints date back to \cite{chopra_partition_1993,deza_clique-web_1992}, and several polytopes related to different kinds of $k$-cut problems of complete graph were proposed at that time. Among them, the $^\leq k$-cut polytope defined in \cite{deza_clique-web_1992} happens to be the image of $\MC^{K_m}$ under a linear transformation since the weighted MAX-$k$-CUT of complete graph $K_m$ can be formulated into MCPP~\eqref{eq:MCPP_pseudo_Bool_def}. (Let $|V|=m$ and $I=\{i^I_1,i^I_2,\dots,i^I_k\}$ for $I\in V$. Note that a $^\leq k$-cut can be expressed by an $x\in \mathcal{X}$, and components of $\chi(S_1,\dots,S_h)$ of \cite{deza_clique-web_1992} can be formulated as $1-\sum_{t=1}^k x_{i^{I_1}_t}x_{i^{I_2}_t}$.) 

%


Very recently, the bipartite Boolean quadric polytope (BQP) with multiple choice constraints \cite{barmann_bipartite_2023}
\begin{equation}
\label{eq:P_G_I_defn}
P(G,\mathcal{I})=\conv\left\{(x,y,z)\in \{0,1\}^{X\cup Y\cup E}\middle|x_iy_j=z_{ij}, \forall\, ij\in E;x\in X^{\mathcal{I}}\right\} 
\end{equation}
was proposed for investigating a variant of bilinear programming 
\begin{subequations}
\label{eq:bilinear_prog_multi_choice}
\begin{align}
\max_{(x,y)\in \{0,1\}^{X\cup Y}} &\quad \sum_{i\in X} a_ix_i+\sum_{j\in Y} a_jy_j+\sum_{\{i,j\}\in E} a_{ij}x_iy_i,\\ 
\text{s.t.} &\quad \sum_{i\in I} x_i\leq 1,\ \forall\, I\in \mathcal{I}, \label{eq:bilinear_prog_multi_choice_1}
\end{align}
\end{subequations}
where $G=(X\cup Y,E)$ is a bipartite graph ($X$ and $Y$ are color classes of $G$), $\mathcal{I}$ is a partition of $X$, and $X^{\mathcal{I}}=\{x\in\{0,1\}^X|\sum_{i\in I} x_i\leq 1,\ \forall I\in \mathcal{I}\}$ is the set of $0$-$1$ vectors that satisfy the multiple choice constraints in Eq.~\eqref{eq:bilinear_prog_multi_choice_1}.
We would like to point out that Eq.~\eqref{eq:bilinear_prog_multi_choice_1} is slightly different from Eq.~\eqref{eq:sum_x_i_1}
for the former are inequalities. Obviously, by introducing $x_{\bar{i}_I}=1-\sum_{i\in I} x_i\in \{0,1\}$ for $I\in \mathcal{I}$, Eq.~\eqref{eq:bilinear_prog_multi_choice_1} becomes $\sum_{i\in I} x_i+x_{\bar{i}_I}=1$, and Problem~\eqref{eq:bilinear_prog_multi_choice} can thus be reformulated as MCPP~\eqref{eq:MCPP_pseudo_Bool_def}.
As a result, when $G$ is subset-uniform, $P(G,\mathcal{I})$ is nothing but a projection of $\MC^H$ (see Example~\ref{eg:BQP_MC_H}) and thus Theorem~\ref{thm:lifted_MP0} can be applied to $P(G,\mathcal{I})$ for achieving some results of \cite{barmann_bipartite_2023}  in a straightforward manner (see  Example~\ref{eg:lifted_barmann})).  

In a word, the proposed $\MC^H$ serves as a polytope associated with general case of $0$-$1$ polynomial programming with (equality) multiple choice constraints and is thus different from those existing polytopes mentioned above. However, those polytopes can be regarded as images of $\MC^{H}$ under linear maps in certain conditions.

\subsection{Organization and convention}

The main task of this paper is to prove Theorems~\ref{thm:alpha_acyclic} and \ref{thm:lifted_MP0}. In Section~\ref{sec:basic}, we give basic properties of $\MC^H$ that are useful for the subsequent investigation. Since $\MC^H$ is not full-dimensional, Section~\ref{sec:affine_hull} provides some techniques to address this deficiency such as considering full-dimensional $\MC^H_\leq$. Then, Section~\ref{sec:decomp} gives a sufficient condition for the decomposability of $\MC^H$ which plays a key role in proving Theorem~\ref{thm:alpha_acyclic}. After that, Sections~\ref{sec:acyclic} and \ref{sec:lifted_MP_H} detail the proofs of Theorems~\ref{thm:alpha_acyclic} and \ref{thm:lifted_MP0}, respectively. Finally, Section~\ref{sec:con} rounds the paper off with some discussions. 

Throughout the rest of this paper, we adopt the following conventions. 
\begin{itemize}
\item Whenever we refer to $\MC^H$, it is assumed that the vertex set of hypergraph $H$ is an partition of $[n]$ and $\MC^H$ is associated with a concrete linearized MCPP instance~\eqref{eq:Linearized_MCPP}.

\item For hypergraphs $H(V,E)$ and $H'(V',E')$, we say that $H$ is a \emph{subhypergraph} of $H'$, denoted by $H\subseteq H'$, if $V\subseteq V'$ and $E\subseteq E'$. 

\item A hypergraph is called \emph{downward-closed}, if for any $e'\subseteq e$ and $e\in E$ with $|e'|>1$, it holds that $e'\in E$.

\item For any finite sets $S$ and $S'\subseteq S$, and $x\in \mathbb{R}^S$, we use $\proj_{S'} x\in\mathbb{R}^{S'}$ to denote the projection mapping that extracts components of $x$ whose indices are in $S'$. It extends to a set $X\subseteq \mathbb{R}^S$, 
\begin{equation}
\proj_{S'} X=\{\proj_{S'} x|x\in X\}.
\end{equation}

\item For any set $S$ and $s \in S$, we use $S-s$ to denote $S\setminus\{s\}$. 

\item We use the shorthand $st$ for an edge $\{s,t\}$ of a graph.

\item We use $\mathscr{L}$ to denote the linearization operator. It applies to a multilinear polynomial (meaning that the polynomial is linear w.r.t. each component) involving $w_i, i\in S$, where $S$ is a finite index set and replaces all the monomials $\prod_{i\in J}w_i$ with $w_{J}$ (after possibly expanding the products). Thus the resulting function is linear to $w_J, J\subseteq S$. For example,
\[
\mathscr{L}(1-w_1)(1-w_2)=1-w_1-w_2+w_{\{1,2\}}.
\]
Note that the entries $J$ are limited in general, e.g., $J\in\mathcal{J}^H$. Hence we have to certify all the monomials every time before applying the operator $\mathscr{L}$. An important property of $\mathscr{L}$ is that the operator does not change the value of the polynomial for any $0$-$1$ vector $w$ satisfying production constraints.

\end{itemize}

\section{Basic properties}
\label{sec:basic}

In this section, we first delineate the subset-uniformity of $\mathcal{J}^H$ and then present some basic properties associated with two kinds of projection of $\MC^H$. Some of the properties are key in the proof of our main theorems.

\subsection{Subset-uniformity of $\mathcal{J}^H$}
\label{subsec:intro_J_H}

%
%

For each nonempty subset $e\subseteq [n]$, let
\begin{equation}
\mathcal{J}^e:=\{J\in \mathcal{J}_\mathscr{P}|E(J)=e\}
\end{equation}
collect all the subsets in $\mathcal{J}_\mathscr{P}$ associated with $e$, and then define 
\begin{equation}
\label{eq:J_H_defn}
\mathcal{J}^H:=\bigcup_{e\in L(V)\cup E}\mathcal{J}^e.
\end{equation}
Noting that $\bigcup_{e\in L(V)} \mathcal{J}^e=L([n])$, the indispensable components $w_i=w_{\{i\}}, i\in [n]$ in Eq.~\eqref{eq:Linearized_MCPP} are covered by $\mathcal{J}^H$.

For any $J\in \mathcal{J}_{\mathscr{P}}$ and $i\in J$, $i,i'\in I$, $I\in V$, it is obvious that $E(J)=E((J - i)\cup\{i'\})$. Since $\mathcal{J}^H$ is consists of groups of $J$s that share the same $E(J)$, we have
\begin{equation}
\label{eq:subset-uniformity}
J\in \mathcal{J}^H\ \text{if and only if}\ (J - i) \cup \{i'\} \in \mathcal{J}^H.
\end{equation}
We call the above property for $\mathcal{J}^H$ \emph{subset-uniformity} w.r.t. $V$. This subset-uniformity is inspired by \cite{barmann_bipartite_2023} and the definitions coincide in the following sense: A bipartite graph $G(X\cup Y,E)$ is subset-uniform w.r.t. the partition $\mathcal{I}$ (of $X$) in the context of \cite{barmann_bipartite_2023} if and only if $E$ is subset-uniform w.r.t. the partition $\mathcal{I}\cup L(Y)$ of $X\cup Y$ in our definition~\eqref{eq:subset-uniformity}. The subset-uniformity allows us to access entries of $w$ by groups of $J\in\mathcal{J}^e$ when $e\in L(V)\cup E$. This is crucial for many formulations of our propositions,  see, e.g., the summation in Eqs.~\eqref{eq:MC_H_lift_add_hyperedge} and \eqref{eq:H_repres_affhull_sym}.
In fact, when some problems with certain symmetry, for instance, MAX-$k$-CUT \cite{shao_ode_2023},  
are rewritten into the MCPP form \eqref{eq:MCPP_pseudo_Bool_def}, the resulting $\mathcal{J}$ is naturally subset-uniform w.r.t. $V$. 
That is, obtaining $\mathcal{J}^H$ from $\mathcal{J}$ only needs adding singleton subsets in $L([n])$.

\subsection{Connections between $\MC^H$ and $\MP^H$}


Arbitrarily choose an $\bar{i}_I\in I$ in order to stand for each $I\in V$ and let $R=\{\bar{i}_I\}_{I\in V}$.
Then there exists a natural one-to-one correspondence between $R$ and $V$, leading to the correspondence between $\mathcal{J}^H|_R:=\{J\in \mathcal{J}^H|J\subseteq R\}$ and $L(V)\cup E$:
\begin{equation}
\label{eq:J_H_R_E_one_to_one}
J\leftrightarrow E(J).
\end{equation}
Accordingly, it is easy to verify that
\begin{equation}
\label{proj0}
\proj_{\mathcal{J}^H|_R} \mathscr{S}^H=\left\{w'\in \{0,1\}^{\mathcal{J}^H|_R} \middle|w'_{J}=\prod_{i\in J} w'_i,\ \forall\, J\in \mathcal{J}^H|_R,|J|>1\right\}.
\end{equation}
Comparing Eqs.~\eqref{eq:multilinear_set} and \eqref{proj0} together, we are able to conclude that there is one-to-one correspondence between $\proj_{\mathcal{J}^H|_R} \mathscr{S}^H$ and $\mathcal{S}^H$. Considering their convex hulls, 
the following connection as stated in Proposition~\ref{prop:MP_H_one_to_one} is achieved. 

\begin{proposition}
\label{prop:MP_H_one_to_one}
Given a hypergraph $H(V,E)$, and $R=\{\bar{i}_I\}_{I\in V}\in \mathcal{J}^V$, 
where $\bar{i}_I\in I$ for $I\in V$, then $\proj_{\mathcal{J}^H|_R}\MC^H$ is in one-to-one correspondence with $\MP^H$, denoted as $w\leftrightarrow z$. The mapping is defined as follows: For all $J\in \mathcal{J}^H|_R$, let $z_{E(J)}=w_J$.
\end{proposition}


\subsection{Connections between $\MC^H$ and $\MC^{H'}$ when $H\subseteq H'$}

When $H$ is a subhypergraph of $H'$, it is obvious that $\mathcal{J}^{H}\subseteq \mathcal{J}^{H'}$.  From the definition~\eqref{eq:def_SG_MCMP_set}, it is also evident that $\proj_{\mathcal{J}^{H}} \mathscr{S}^{H'}=\mathscr{S}^{H}$. Considering their convex hulls, we arrive at the following proposition.

\begin{proposition}
\label{prop:proj_subhgraph_MC_H}
Suppose $H$ is a subhypergraph of $H'$, then
\begin{equation}
\proj_{\mathcal{J}^{H}} \MC^{H'}=\MC^{H}.
\end{equation}
\end{proposition}

When we add a hyperedge which is contained in another hyperedge into $H$ and the resulting hypergraph is $H'$, there exists an isomorphism 
from $\MC^{H'}$ to $\MC^H$ as shown in Proposition~\ref{prop:add_hyperedge_isomorph}. This result also serves as a technical lemma for proving Theorem~\ref{thm:alpha_acyclic}.



\begin{proposition}
\label{prop:add_hyperedge_isomorph}
Given a hypergraph $H(V,E)$, let $e_1\in E$ and $e_2\subset e_1$ such that $e_2\notin E$ and $|e_2|>1$. Define $H'=(V,E\cup \{e_2\})$. Then, the projection of $\MC^{H'}$ onto $\MC^H$ is a linear isomorphism, and
\begin{equation}
\label{eq:MC_H_lift_add_hyperedge}
\MC^{H'}=\left\{w\in \mathbb{R}^{\mathcal{J}^{H'}}\middle|\proj_{\mathcal{J}^{H}} w\in \MC^H;w_J=\sum_{J'\in \mathcal{J}^{e_1}: J'\supseteq J} w_{J'},\ \forall\, J\in \mathcal{J}^{e_2}\right\}.
\end{equation}
\end{proposition}

\begin{proof}
Define $\mathcal{T}:\MC^{H}\rightarrow \MC^{H'}$ that linearly maps $v$ to $w$ as follows: For any $J\in \mathcal{J}^{e_2}$,
\begin{equation}
w_J=\sum_{J'\in \mathcal{J}^{e_1}: J'\supseteq J} v_{J'};
\end{equation}
and for $J\in \mathcal{J}^{H}$, $w_J=v_J$. Obviously, $\proj_{\mathcal{J}^{H}} w=v$.

When $v\in \mathscr{S}^{H}$, we claim that $w=\mathcal{T}(v)\in \mathscr{S}^{H'}$. It suffices to prove for $J\in \mathcal{J}^{H'} \setminus \mathcal{J}^{H}=\mathcal{J}^{e_2}$,
\begin{equation}
w_J=\prod_{i\in J} w_i.
\end{equation}
It holds since
\begin{subequations}
\begin{align*}
w_J &= \sum_{J'\in \mathcal{J}^{e_1}: J'\supseteq J} v_{J'}\\
	&= \sum_{J'\in \mathcal{J}^{e_1}: J'\supseteq J} \prod_{i\in J'} v_i\\
	&= \prod_{i\in J} v_i \sum_{J'\in \mathcal{J}^{e_1}: J'\supseteq J} \prod_{i\in J'\setminus J} v_i\\
	&= \prod_{i\in J} v_i \prod_{I\in e_1\setminus e_2} \sum_{i\in I} v_i\\
	&= \prod_{i\in J} v_i,
\end{align*}
\end{subequations}
and $w_i=v_i, i\in [n]$. Thus, $\mathcal{T}$ gives a one-to-one correspondence between $\mathscr{S}^H$ and $\mathscr{S}^{H'}$. Moreover, since $\mathcal{T}$ is linear, it it naturally extended to form a one-to-one correspondence between $\MC^H$ and $\MC^{H'}$. Hence, we have
\begin{subequations}
\begin{align}
\MC^{H'} & = \mathcal{T}(\MC^H)\\
&= \left\{w\in \mathbb{R}^{\mathcal{J}^{H'}}\middle|\proj_{\mathcal{J}^{H}} w\in \MC^H;w_J=\sum_{J'\in \mathcal{J}^{e_1}: J'\supseteq J} w_{J'},\ \forall\, J\in \mathcal{J}^{e_2}\right\}.
\end{align}
\end{subequations}
\end{proof}

\section{The affine hull of $\MC^{H}$}
\label{sec:affine_hull}
Given a hypergraph $H(V,E)$ and $D=\{\bar{i}_I\}_{I\in V}\in\mathcal{J}^V$,  
let
\begin{align}
\mathcal{J}^H_\leq(D) & := \mathcal{J}^H|_{[n]\setminus D}=\left\{J\in \mathcal{J}^H\middle|J\subseteq [n]\setminus D\right\}, \label{eq:MCPP_J_leq_def} \\
\MC^H_\leq(D) & := \conv \bigg\{v\in \{0,1\}^{\mathcal{J}^H_\leq(D)} \bigg|
v(I\setminus D)\leq 1,\ \forall\, I\in V; \nonumber \\
&\hspace{11.35em} v_{J}=\prod_{i\in J} v_i,\ \forall\, J\in \mathcal{J}^H_\leq(D), |J|>1\bigg\}. \label{eq:MCPP_MC_leq_def}
\end{align}
When $D$ is clear from the context, we often simply write $\mathcal{J}^H_\leq$ and $\MC^H_\leq$ for $\mathcal{J}^H_\leq(D)$ and $\MC^H_\leq(D)$, respectively.

\begin{example}
\label{eg:MC_H_leq_iso_MP_H}
If $|I|=2$ for all $I\in V$, letting $R=[n]\setminus D$, we have $R\in \mathcal{J}^V$. Then $\mathcal{J}^H_\leq = \mathcal{J}^H|_R$ and $\MC^H_\leq = \proj_{\mathcal{J}^H|_R} \MC^H$. According to Proposition~\ref{prop:MP_H_one_to_one}, we have $\MC^H_\leq$ is in one-to-one correspondence with $\MP^H$.
\end{example}


\begin{example}
\label{eg:BQP_MC_H}
Suppose $G(X\cup Y,E)$ is a subset-uniform bipartite graph in the context of \cite{barmann_bipartite_2023}. Then $P(G,\mathcal{I})$ is actually a special case of $\MC^H_\leq$, which can be verified as follows. 
The subset-uniformity requires $G$ to have the following properties: For any $I\in \mathcal{I}$ and $i_1,i_2\in I$, they share the same neighborhood in $E$. Let $\bar{i}_I, I\in \mathcal{I}\cup L(Y)$ be arbitrary indices other than $X\cup Y$, and define
\begin{align}
V&=\{I\cup \{\bar{i}_I\}\}_{I\in \mathcal{I}\cup L(Y)},
\label{eq:V_from_P_G_I}\\
\hat{E}&=\{I_1I_2|\exists i_1\in I_1, i_2\in I_2\ \text{s.t.}\ i_1i_2\in E\},\label{eq:hat_E_from_P_G_I}\\
D&=\{\bar{i}_I\}_{I\in \mathcal{I}\cup L(Y)}.
\label{eq:D_from_P_G_I}
\end{align}
It is readily seen that $H=(V,\hat{E})$ is isomorphic to $\mathcal{G}=(\mathcal{I}\cup Y,\mathcal{E})$ which is achieved by merging the vertices in each subset of the partition $\mathcal{I}$ into a single vertex. Additionally, $\mathcal{J}^H_\leq(D)=L(X)\cup L(Y)\cup E$. Comparing Eqs.~\eqref{eq:MCPP_MC_leq_def} and \eqref{eq:P_G_I_defn}, we have that $\MC^H_\leq(D)$ is in one-to-one correspondence with $P(G,\mathcal{I})$.
\end{example}

Let $v^U\in\{0,1\}^{\mathcal{J}^H_\leq} $ denote the vector with entries $1$ corresponding to $J\subseteq U $ and $0$ for others. Take $v^U,U\in\{\varnothing\}\cup \mathcal{J}^H_\leq$. It is straightforward to verify that they are $|\mathcal{J}^H_\leq|+1 $ affinely independent vectors of $\MC^H_\leq$, implying: 
\begin{proposition}
\label{prop:MCPP_polytope_full_dim}
$\MC^H_\leq(D)$ is full-dimensional.
\end{proposition}

Due to the equality constraints $w(I)=1$,  $\MC^H$ is never full-dimensional, and thus its affine hull is strictly contained in $\mathbb{R}^{\mathcal{J}^H}$. With the help of $\MC^H_\leq$, the half-space representation of $\aff\MC^H$ can be reached when $H$ is downward-closed (see Theorem~\ref{thm:MCPP_polytope_aff_MC_H}). Comparing Eqs.~\eqref{eq:MCPP_MC_leq_def} with \eqref{eq:def_SG_MCMP_set}, it is easily seen that $\proj_{\mathcal{J}^H_\leq} \MC^H = \MC^H_\leq$. Denote the restriction of this projection to $\MC^H$ by 
\begin{equation}
\label{proj}
\proj_{\leq}:\MC^H\rightarrow \MC^H_\leq. 
\end{equation}
If $\proj_{\leq}$ is invertible, then it defines a linear isomorphism between $\MC^H$ and the full-dimensional $\MC^H_\leq$. Although this is not always true, 
we are able to give a sufficient condition (see Proposition~\ref{prop:MC_leq_bijection}).
This isomorphism allows us to verify whether a given inequality is facet-inducing for $\MC^H$ in an easier way, see Section~\ref{sec:lifted_MP_H}.

 

\begin{proposition}
\label{prop:MC_leq_bijection}
Suppose $H(V,E)$ is downward-closed. Given any $D\in \mathcal{J}^V$, $\proj_{\leq}$ in Eq.~\eqref{proj} defines a bijection between $\MC^H$ and $\MC^H_\leq$.
\end{proposition}

\begin{proof}
We define $\mathcal{T}:\MC^H_\leq\rightarrow \MC^H$ that maps $v$ to $w$ as follows: 
\begin{equation}
\label{eq:T_mapping_part}
w_J=v_J,\quad\forall\, J\in \mathcal{J}^H_\leq, 
\end{equation}
and for other components,
\begin{equation}
\label{eq:T_mapping_subset}
w_{J} = \mathscr{L} \prod_{\bar{i}_I\in J\cap D}\left(1-\sum_{i\in I-\bar{i}_I} v_i\right)\prod_{i\in J\setminus D} v_i, \quad J \in \mathcal{J}^H\setminus \mathcal{J}^H_\leq.
\end{equation}
Note that any monomial $\prod_{i\in J'} v_i$ in the expansion of
\[
\prod_{\bar{i}_I\in J\cap D}\left(1-\sum_{i\in I-\bar{i}_I} v_i\right)\prod_{i\in J\setminus D} v_i
\]
satisfies $E(J')\subseteq E(J)\in L(V)\cup E$ and $J'\cap D=\varnothing$. Combining the downward-closedness of $H$, we have $J'\in \mathcal{J}^H_\leq$. 
Therefore, $\mathscr{L}$ is well-defined, and so is $\mathcal{T}$.

Now, we verify that $\mathcal{T}$ is the inverse mapping of $\proj_{\leq}$.
It is clear from \eqref{eq:T_mapping_part} that $\proj_{\leq} \circ \mathcal{T}=id$. Next, we prove $\mathcal{T}\circ \proj_\leq=id$. For any $w\in \mathscr{S}^H$, suppose $v=\proj_{\leq} w$, then $v\in \{0,1\}^{\mathcal{J}^H_\leq}$ satisfies the production constraints. Hence, the values of any polynomial w.r.t. $v$ remain unchanged under the linearization operator.

We prove that $\mathcal{T}_J(v)=w_J$ for any $J\in \mathcal{J}^H$. Firstly, for any $J\in \mathcal{J}^H_\leq$, we have
\[
\mathcal{T}_J(v)=v_J=w_J.
\]
Especially, when $i\in [n]\setminus D$, $v_i=w_i$. Secondly, for any $J\in \mathcal{J}^H\setminus \mathcal{J}^H_\leq$, we have
\begin{subequations} \label{eq:v_i_to_w_i}
\begin{align}
\mathcal{T}_J(v)&=\mathscr{L} \prod_{\bar{i}_I\in J\cap D}\left(1-\sum_{i\in I-\bar{i}_I} v_i\right)\prod_{i'\in J\setminus D} v_{i'}\\
&=\prod_{\bar{i}_I\in J\cap D}\left(1-\sum_{i\in I-\bar{i}_I} v_i\right)\prod_{i'\in J\setminus D} v_{i'}.
\end{align}
\end{subequations}
Using the fact that $i\in I-\bar{i}_I\subseteq [n]\setminus D$ and $i'\in J\setminus D\subseteq [n]\setminus D$, we can replace $v_i$ and $v_{i'}$ with $w_i$ and $w_{i'}$, respectively, and obtain
\begin{subequations}
\begin{align}
\text{RHS of Eq.~\eqref{eq:v_i_to_w_i}} &=\prod_{\bar{i}_I\in J\cap D}\left(1-\sum_{i\in I-\bar{i}_I} w_i\right)\prod_{i'\in J\setminus D} w_{i'}\\
&=\prod_{\bar{i}_I\in J\cap D}w_{\bar{i}_I}\prod_{i'\in J\setminus D} w_{i'}\\
&=w_J.
\end{align}
\end{subequations}
Therefore, $\mathcal{T}(\proj_\leq (w))=\mathcal{T}(v)=w$. Combining the fact that $w$ is any vertex of $\MC^H$, and that both $\mathcal{T}$ and $\proj_\leq$ are linear, we conclude that $\mathcal{T}\circ \proj=id$ also holds on $\MC^H$. That is, $\mathcal{T}=\proj_\leq^{-1}$. 
\end{proof}

The projection $\proj_\leq$ in Eq.~\eqref{proj} also defines an isomorphism between each facet of $\MC^H$ and its corresponding facet of $\MC^H_\leq$. Consider a facet-inducing inequality for $\MC^H$, $a\cdot w\leq \delta$, which induces facet $F$. Suppose the nonzero components of $a$ corresponds only to $J\in \mathcal{J}^H_\leq$, then it can be regarded as an inequality in the space of $\MC^H_\leq$, i.e.,
\begin{equation}
\label{eq:valid_ineq_MC_leq0}
\proj_{\mathcal{J}^H_\leq}{a}\cdot v\leq \delta.
\end{equation}
For any $v\in \MC^H_\leq$, we have $\proj_{\mathcal{J}^H_\leq}{a}\cdot v=a\cdot \proj_\leq^{-1}(v)$. Thus Eq.~\eqref{eq:valid_ineq_MC_leq0} is also valid for $\MC^H_\leq$. Moreoever, it induces a facet, $\proj_\leq F$. It is easily seen that the inverse proposition of facet of $\MC^H_\leq$ is also true, leading to Corollary~\ref{corol:facet_MC_H_leq}.

\begin{corollary}
\label{corol:facet_MC_H_leq}
Suppose $H(V,E)$ is downward-closed. $\MC^H$ is defined w.r.t. $D\in \mathcal{J}^V$. Let $a\cdot w\leq \delta$ be a valid inequality for $\MC^H$ such that
\begin{equation}
\{J\in\mathcal{J}^H|a_J\neq 0\}\subseteq \mathcal{J}^H_\leq.
\end{equation}
Then $a\cdot w\leq \delta$ is facet-inducing for $\MC^H$ if and only if $\proj_{\mathcal{J}^H_\leq}a\cdot v\leq \delta$ is facet-inducing for $\MC_\leq^H$.
\end{corollary}

Proposition~\ref{prop:MCPP_polytope_full_dim} implies $\aff \MC_{\leq }^H=\mathbb{R}^{\mathcal{J}^H_\leq }$. The inverse linear mapping $\proj_\leq^{-1}$ defined above can be naturally extended to a linear map $\tilde{\mathcal{T}}$ on $\mathbb{R}^{\mathcal{J}^H_\leq}$ such that its definition is still consistent with Eqs.~\eqref{eq:T_mapping_part} and \eqref{eq:T_mapping_subset}. 
In consequence, we have 
 \begin{equation}
\label{3.16}
\tilde{\mathcal{T}}(\mathbb{R}^{\mathcal{J}^H_\leq })=\tilde{\mathcal{T}}(\aff \MC_{\leq }^H)=\aff \tilde{\mathcal{T}}(\MC_{\leq }^H)=\aff \MC^H.
\end{equation}
From this, the half-space representation of $\aff \MC^H$ can be obtained.

\begin{theorem}
\label{thm:MCPP_polytope_aff_MC_H}
Suppose $H(V,E)$ is downward-closed, and $D=\{\bar{i}_I\}_{I\in V}\in \mathcal{J}^V$, then
\begin{equation}
\label{eq:H_repres_affhull}
\aff \MC^H=\left\{w\in \mathbb{R}^{\mathcal{J}^H}\middle| w_{J} = \mathscr{L} \prod_{\bar{i}_I\in J\cap D}\left(1-\sum_{i\in I-\bar{i}_I} w_i\right)\prod_{i\in J\setminus D} w_i, \ \forall\, J \in \mathcal{J}^H\setminus \mathcal{J}^H_\leq\right\}.
\end{equation}
\end{theorem}
\begin{proof}
We define $\tilde{\mathcal{T}}:\mathbb{R}^{\mathcal{J}^H_\leq}\rightarrow\mathbb{R}^{\mathcal{J}^H}$ that maps $v$ to $w$ as
\begin{equation}
w_J=v_J,\quad\forall\, J\in \mathcal{J}^H_\leq, 
\end{equation}
and for other components,
\begin{equation}
\label{eq:tilde_T_mapping_subset}
w_{J} = \mathscr{L} \prod_{\bar{i}_I\in J\cap D}\left(1-\sum_{i\in I-\bar{i}_I} v_i\right)\prod_{i\in J\setminus D} v_i, \quad J \in \mathcal{J}^H\setminus \mathcal{J}^H_\leq.
\end{equation}
Since the expressions are the same to Eqs.~\eqref{eq:T_mapping_part} and \eqref{eq:T_mapping_subset}, $\tilde{\mathcal{T}}$ is extension of $\mathcal{T}$ defined in the proof of Proposition~\ref{prop:MC_leq_bijection}, and the RHS of Eq.~\eqref{eq:tilde_T_mapping_subset} is linear w.r.t. components of $w$ corresponding to indices in $\mathcal{J}^H_\leq$. Thus $\tilde{\mathcal{T}}(\MC_{\leq }^H)=\mathcal{T}(\MC^H_\leq)=\MC^H$, and
Eq.~\eqref{3.16} holds. 
Substituting $v_{J'}$ by $w_{J'}$ for all $J'\in \mathcal{J}^H_\leq$ in Eq.~\eqref{eq:tilde_T_mapping_subset}, 
we obtain
\begin{equation}
\tilde{\mathcal{T}}(\mathbb{R}^{\mathcal{J}^H_\leq })=
\left\{w\in \mathbb{R}^{\mathcal{J}^H}\middle|w_{J} = \mathscr{L} \prod_{\bar{i}_I\in J\cap D}\left(1-\sum_{i\in I-\bar{i}_I} w_i\right)\prod_{i\in J\setminus D} w_i, \ \forall\, J \in \mathcal{J}^H\setminus \mathcal{J}^H_\leq\right\}.
\end{equation}
The proof is thus completed using Eq.~\eqref{3.16}. 
\end{proof}

Note that the RHS of Eq.~\eqref{eq:H_repres_affhull} depends on $D$ while $\aff\MC^H$ is independent of $D$. In fact, there exists an equivalent half-space representation of $\aff\MC^H$ which remains unchanged for any permutation within $I\in V$. The following corollary can be verified from Theorem.~\ref{thm:MCPP_polytope_aff_MC_H}.
\begin{corollary}
Suppose $H(V,E)$ is downward-closed, then
\begin{equation}
\label{eq:H_repres_affhull_sym}
\begin{aligned}
\aff \MC^H=
&\bigg\{w\in \mathbb{R}^{\mathcal{J}^H}\bigg|w(I)=1,\ \forall\, I\in V;\\
& w_J=\sum_{J'\in \mathcal{J}^{e'}: J'\supset J}w_{J'},\ \forall\, e,e'\in L(V)\cup E, e\subset e',|e'|=|e|+1,\forall\, J\in \mathcal{J}^e \bigg\}.
\end{aligned}
\end{equation}
\end{corollary}

\section{Decomposability of $\MC^{H}$}
\label{sec:decomp}
In this section, we provide a sufficient condition for the decomposability of the MCPP polytope $\MC^{H}$. Let $H_1(V_1,E_1)$ and $H_2(V_2,E_2)$ be subhypergraphs of $H$ such that $H_1\cup H_2=H$, where $H_1\cup H_2$ is defined as $(V_1\cup V_2,E_1\cup E_2)$. We say that $\MC^H$ is decomposable into $\MC^{H_1}$ and $\MC^{H_2}$ if
\begin{equation}
\label{eq:MCPP_polytope_decomp_def}
\MC^H=\overline{\MC}^{H_1}\cap \overline{\MC}^{H_2},
\end{equation}
where the polyhedron $\overline{\MC}^{H_k}$ collects all the points in $\mathbb{R}^{\mathcal{J}^H}$ whose projections on $\mathbb{R}^{\mathcal{J}^{H_k}}$ belong to $\MC^{H_k}$, $k=1, 2$. 
Namely,
\begin{equation}
\overline{\MC}^{H_k}=\left\{w\in \mathbb{R}^{\mathcal{J}^{H}}\middle|\proj_{\mathcal{J}^{H_k}}w\in \MC^{H_k}\right\},\quad k=1,2.
\end{equation}
The decomposability in Eq.~\eqref{eq:MCPP_polytope_decomp_def} implies that the half-space representation of $\MC^H$ is the union of the  half-space representations of $\MC^{H_1}$ and $\MC^{H_2}$. Therefore, the study of $\MC^H$ can be transferred to that of smaller $\MC^{H_k}$. In the rest of this paper, for $w\in \mathbb{R}^{\mathcal{J}^H}$ and $H'$ being a subhypergraph of $H$, we shortly denote $\proj_{H'} w=\proj_{\mathcal{J}^{H'}}w$.


\begin{example}
\label{eg:simple_decomp}
A natural sufficient condition for the decomposability is $V_1\cap V_2=\varnothing$, in which any $J_1\in \mathcal{J}^{H_1}$ and $J_2\in \mathcal{J}^{H_2}$ are disjoint. In the definition of $\mathscr{S}^H$ in Eq.~\eqref{eq:def_SG_MCMP_set}, all production constraints $w_{J}=\prod_{i\in J} w_i$ are partitioned into two unrelated groups. Hence
\begin{equation}
\mathscr{S}^H=\left\{w\in \mathbb{R}^{\mathcal{J}^{H}}\middle|\proj_{H_k}w\in \mathscr{S}^{H_k},\, k=1,2\right\}.
\end{equation}
Then it is easy to see that $\MC^H=\overline{\MC}^{H_1}\cap \overline{\MC}^{H_2}$.
\end{example}


Due to the redundancy or non-full-dimensionality inside $\MC^H$ discussed in the previous section, 
the conditions for the decomposability may be weaker than that for $\MP^H$ \cite{del_pia_decomposability_2018}.

\begin{theorem}
\label{thm:decomp}
Suppose the hypergraphs $H(V,E)$, $H_1(V_1,E_1)$ and $H_2(V_2,E_2)$ satisfy $H=H_1\cup H_2$. If $V_1\cap V_2\in (L(V_1)\cup E_1)\cap (L(V_2)\cup E_2)$, then
\begin{equation}
\MC^H=\overline{\MC}^{H_1}\cap \overline{\MC}^{H_2}.
\end{equation}
\end{theorem}

\begin{proof}
Take any $w\in \mathscr{S}^{H}$. It is easily seen that
\[
\proj_{H_k} w\in \mathscr{S}^{H_k}\subset \MC^{H_k},\ k=1,2,
\]
implying $w\in \overline{\MC}^{H_1}\cap \overline{\MC}^{H_2}$. By the arbitrariness of $w$, we have
\begin{equation}
\mathscr{S}^{H}\subseteq \overline{\MC}^{H_1}\cap \overline{\MC}^{H_2}.
\end{equation}
As the RHS of the above equation is convex, we can take the convex hull on the LHS, yielding
\begin{equation}
\MC^H\subseteq \overline{\MC}^{H_1}\cap \overline{\MC}^{H_2}.
\end{equation}

Now it suffices to prove the inverse inclusion.
Since $\mathcal{J}^H=\mathcal{J}^{H_1}\cup\mathcal{J}^{H_2}$, it is easily shown that $\overline{\MC}^{H_1}\cap \overline{\MC}^{H_2}$ is bounded and thus a polytope. 
Suppose $w$ is a vertex of $\overline{\MC}^{H_1}\cap \overline{\MC}^{H_2}$, and let $u^k=\proj_{H_k} w$, $k=1,2$. Since both $\overline{\MC}^{H_1}$ and $\overline{\MC}^{H_2}$ are rational polyhedra, $w$ is also rational. 
Then, by the definition of $\overline{\MC}^{H_k}$, we know that $u^k\in \MC^{H_k}$, so it can be expressed as convex combination of points in $\mathscr{S}^{H_k}$. Since $u^k$ and all points in $\mathscr{S}^{H_k}$ are rational, the convex combination can be written as follows without loss of generality:
\begin{equation}
u^k=\frac{1}{s_k}\sum_{\ell=1}^{s_k} u^{k,\ell},
\end{equation}
where $u^{k,i}\in \mathscr{S}^{H_k},i=1,2,\dots,s_k$. Furthermore, by repeating $u^{k,\ell}$, we can assume without loss of generality that $s_1=s_2=:s$. Let $p=V_1\cap V_2$, and due to the condition, $p\in (L(V_1)\cup E_1)\cap (L(V_2)\cup E_2)$. For any $J\in \mathcal{J}^p$, we have $u^1_{J}=u^2_J=w_J$, and combining that $u^{k,\ell}_J\in\{0,1\}$, we obtain
\begin{equation}
sw_J=su^k_J=\sum_{\ell=1}^s u^{k,\ell}_J=|\{\ell\in [s]|u^{k,\ell}_J=1\}|,\quad k=1, 2.
\end{equation}
Since $\sum_{J\in \mathcal{J}^p} u^{k,\ell}_J=1$, the subsets of $[s]$, $\{\ell\in [s]|u^{k,\ell}_J=1\}$, form a disjoint partion of $[s]$ over $J$ in $\mathcal{J}^p$. 
Note that both of $|\{\ell\in [s]|u^{k,\ell}_J=1\}|$ are equal to $sw_J$. Thus, we can rearrange $\{u^{1,\ell}\}_{\ell\in [s]}$ and assume that $u^{1,\ell}_J=u^{2,\ell}_J,\forall\, J\in \mathcal{J}^p,\ell=1,2,\dots,s$. Then, by the properties of $\mathscr{S}^{H_k}$, for any $i\in \bigcup_{I\in p} I$, we have
\begin{equation}
u^{k,\ell}_i=\sum_{J\in \mathcal{J}^p: J\ni i}u^{k,\ell}_J,
\end{equation}
so $u^{1,\ell}_i=u^{2,\ell}_i$. Moreover, for any $J\in \mathcal{J}^{H_1}\cap \mathcal{J}^{H_2}\subseteq
\mathcal{J}^p$, since $J\subseteq \bigcup_{I\in p} I$, we have
\begin{equation}
u^{1,\ell}_J=\prod_{i\in J}u^{1,\ell}_i=\prod_{i\in J}u^{2,\ell}_i=u^{2,\ell}_J.
\end{equation}
Therefore, we can naturally combine $u^{1,\ell}$ and $u^{2,\ell}$ to form $w^\ell\in \mathscr{S}^H$ for $\ell \in [s]$, and
\begin{equation}
w=\frac{1}{s}\sum_{\ell=1}^{s} w^{\ell}.
\end{equation}
This shows that $w$ is a convex combination of points in $\MC^H$, and hence $w\in \MC^H$. By the arbitrariness of vertex $w$, $\overline{\MC}^{H_1}\cap \overline{\MC}^{H_2}\subseteq \MC^H$.
\end{proof}

\begin{remark}
The above vertices matching technique, expressing a convex combination in terms of $\frac{1}{s}\sum_{\ell=1}^{s} u^{k,\ell}$
and matching the vertices $\{u^{k,\ell}\}_{\ell=1}^s$ between $k=1$ and $k=2$, is inspired by the proof of Theorem 25.1 in \cite{schrijver_combinatorial_2002}. In fact, using this technique to prove Theorem 5 of \cite{del_pia_multilinear_2018} and Theorem 4 of \cite{del_pia_polynomial-size_2023} is much simpler than the original approaches --- calculating all the coefficients of the convex combination. 
\end{remark}


\section{Proof of Theorem~\ref{thm:alpha_acyclic}}
\label{sec:acyclic}
Recall that the join tree is an important characterization for $\alpha$-acyclicity.
\begin{theorem}[Part of Theorem 3.4 of \cite{beeri_desirability_1983}]
\label{thm:alpha_join_tree}
A hypergraph is $\alpha$-acyclic if and only if it has a join tree.
\end{theorem}
We repeat the definition of join tree to be more self-contained.
\begin{definition}[Condition 3.9 of \cite{beeri_desirability_1983}]
Give a hypergraph $H(V,E)$, a join tree $T$ of $H$ is a tree (connected acyclic undirected graph) with vertex set $E$ such that for any path, the intersection of the two endpoints,  is contained in every vertex along the path. More precisely, let $e_1$ and $e_2$ be any distinct vertices of $T$, then each vertex $e$ on the unique path connecting $e_1$ and $e_2$ satisfies $e_1\cap e_2\subseteq e$.
\end{definition}
Note that each vertex of $T$ is a hyperedge of $H$. To distinguish it from vertex of $H$, we sometimes refer to vertices of $T$ as nodes when we need to stress their properties about $T$.


Before showing the proof, we quickly verify that the RHS of Eq.~\eqref{eq:MC_H_repres_alpha_acyclic} is a relaxation of $\MC^H$.

Suppose $H$ is an arbitrary graph. It is easily seen that for any $w\in \mathscr{S}^H$, $e\in E$, and $i_0\in I_0\in e$, we have
\begin{subequations}
\begin{align}
\sum_{J\in \mathcal{J}^{e}: J\ni i_0}w_J &= \sum_{J\in \mathcal{J}^{e}: J\ni i_0}\prod_{i\in J}w_i\\
&= w_{i_0}\prod_{I\in e-I_0}\sum_{i\in I} w_i\\
&=w_{i_0}.
\end{align}
\end{subequations}
Similarly, for any $e,e'\in E$ with $|e\cap e'|>1$ and $J_0\in \mathcal{J}^{e\cap e'}$, it holds that
\begin{equation}
\sum_{J\in \mathcal{J}^e: J\supseteq J_0} w_J =\prod_{i\in J_0} w_i=\sum_{J\in \mathcal{J}^{e'}: J\supseteq J_0}w_J.
\end{equation}
Thus, we can define the following relaxation of $\MC^H$:\begin{subequations}
\label{eq:MC_H_relaxation_intersect}
\begin{align}
\MC^H_{\cap} :=\bigg\{w\in \mathbb{R}_{\geq 0}^{\mathcal{J}^H}\bigg|
&w(I)=1,\ \forall\, I\in V;\label{eq:MC_H_relax_inters_sum_I}\\
&w_i-\sum_{J\in \mathcal{J}^e: J\ni i}w_J=0,\ \forall\, e\in E, i\in I\in e;\\
&\sum_{J\in \mathcal{J}^e: J\supseteq J_0} w_J -\sum_{J\in \mathcal{J}^{e'}: J\supseteq J_0}w_J=0,\ \forall\, e,e'\in E,|e\cap e'|>1,J_0\in \mathcal{J}^{e\cap e'}\bigg\}.\label{eq:equal_inters_main}
\end{align}
\end{subequations}
Since $\mathscr{S}^H\subseteq \MC^H$, we have $\MC^H\subseteq \MC^H_{\cap}$.

When $H$ is $\alpha$-acyclic, by Theorem~\ref{thm:alpha_join_tree}, $H$ admits a join tree. Let $T=(E,\mathcal{E})$ be its join tree. By limiting $e,e'$ adjacent in $T$ from the system of equalities Eq.~\eqref{eq:equal_inters_main}, $\MC^H_{\cap}$ can be further relaxed into
\begin{subequations}
\label{eq:MC_H_relax_join_tree}
\begin{align}
\MC^H_T:=\bigg\{w\in \mathbb{R}_{\geq 0}^{\mathcal{J}^H}\bigg|
&w(I)=1,\ \forall\, I\in V;\label{eq:H_repres_sum_I}\\
&w_i -\sum_{J\in \mathcal{J}^e: J\ni i}w_J=0,\ \forall\, e\in E, i\in I\in e;\label{eq:equal_node_J_e}\\
&\sum_{J\in \mathcal{J}^e: J\supseteq J_0} w_J -\sum_{J\in \mathcal{J}^{e'}: J\supseteq J_0}w_J=0,\ \forall\, ee'\in \mathcal{E},|e\cap e'|>1,J_0\in \mathcal{J}^{e\cap e'}\bigg\},\label{eq:equal_intersect_tree}
\end{align}
\end{subequations}
which is the same to the RHS of Eq.~\eqref{eq:MC_H_repres_alpha_acyclic} of Theorem~\ref{thm:alpha_acyclic}. 
It is obvious that $\MC^H_{\cap}\subseteq \MC^H_T$. Note that the join tree of $H$ is not unique, so $\MC^H_T$ may vary according to $T$. However, the following lemma demonstrates that $\MC^H_T$ is in fact independent of the choice of $T$.

\begin{lemma}
\label{lem:MC_H_cap_eq_T}
Suppose $H=(V,E)$ is $\alpha$-acyclic, and $T=(E,\mathcal{E})$ is a join tree of $H$, then $\MC^H_\cap=\MC^H_T$.
\end{lemma}

\begin{proof}
Comparing Eq.~\eqref{eq:MC_H_relaxation_intersect} with Eq.~\eqref{eq:MC_H_relax_join_tree}, the only difference is that all equalities in Eq.~\eqref{eq:equal_intersect_tree} are contained in Eq.~\eqref{eq:equal_inters_main}, thus $\MC^H_T\supseteq \MC^H_\cap$.
To prove $\MC^H_T=\MC^H_\cap$, it suffices to show that all equalities in Eq.~\eqref{eq:equal_inters_main} are implicit equalities of $\MC^H_T$. 

Suppose $e,e'\in E$, $|e\cap e'|>1$, 
then there exists a unique path from $e$ to $e'$ in $T$. 
We prove by induction on the path length $l$ that for any $J_0\in \mathcal{J}^{e\cap e'}$, the following equality holds for any $w\in \MC^H_T$:
\begin{equation}
\label{eq:inters_implicit_equal}
\sum_{J\in \mathcal{J}^e: J\supseteq J_0} w_J =\sum_{J\in \mathcal{J}^{e'}: J\supseteq J_0}w_J.
\end{equation}


When $l\leq 1$, the statement is trivial since $l=0$ implies $e=e'$ and $l=1$ implies $ee'\in \mathcal{E}$, meaning that Eq.~\eqref{eq:inters_implicit_equal} is just Eq.~\eqref{eq:equal_intersect_tree}.

Suppose $e_1,e_2\in E$ with $|e_1\cap e_2|>1$ and the length of the path from $e_1$ to $e_2$ in $T$ is $l$. Now assume the statement holds for $e,e'$ such that the path length is shorter that $l$. Note that Eq.~\eqref{eq:inters_implicit_equal} holds even when $|e\cap e'|\leq 1$ for the following reasons: If $|e\cap e'|=1=|J_0|$, according to equality Eq.~\eqref{eq:equal_node_J_e}, it is easy to see that both sides of Eq.~\eqref{eq:inters_implicit_equal} are equal to $w_{J_0}$; If $|e\cap e'|=0$, then $J_0=\varnothing$, 
and combining Eqs.~\eqref{eq:H_repres_sum_I} and \eqref{eq:equal_node_J_e}, both sides of Eq.~\eqref{eq:inters_implicit_equal} are equal to $1$. Therefore, we can assume Eq.~\eqref{eq:inters_implicit_equal} holds without checking if $|e\cap e'|>1$.

Take an internal node $e_0$ along the path from $e_1$ to $e_2$, then the length of the path from $e_1$ to $e_0$ in $T$ is strictly less than $l$. By the induction hypothesis, for any $J'\in \mathcal{J}^{e_1\cap e_0}$, the following equality holds for $w\in\MC^H_T$:
\begin{equation}
\label{eq:inters_implicit_assump1}
\sum_{J\in \mathcal{J}^{e_1}: J\supseteq J'} w_J =\sum_{J\in \mathcal{J}^{e_0}: J\supseteq J'}w_J.
\end{equation}
According to the properties of the join tree, $e_1\cap e_2\subseteq e_0$, so $e_1\cap e_2\subseteq e_1\cap e_0$. Given $J_0\in\mathcal{J}^{e_1\cap e_2}$, consider the index set of summation for $J$ on the left side of Eq.~\eqref{eq:inters_implicit_equal} for $e\leftarrow e_1$, that is
\begin{equation}
\mathcal{J}^{e_1}(J_0):=\left\{J\in \mathcal{J}^{e_1}\middle|J\supseteq J_0\right\}.
\end{equation}
Suppose $J\in \mathcal{J}^{e_1}(J_0)$. Let $J_1=J\cap \bigcup_{I\in e_1\cap e_0} I$, then $J_1\in\mathcal{J}^{e_1\cap e_0}$. Since $e_1\cap e_2\subseteq e_1\cap e_0$, we have $J_1\supseteq J_0$. Moreover, as $J$ ranges over $\mathcal{J}^{e_1}(J_0)$, the corresponding $J_1$ ranges over $\{J'\in \mathcal{J}^{e_1\cap e_0}|J'\supseteq J_0\}$. Thus elements $J$s of $\mathcal{J}^{e_1}(J_0)$ can be partitioned according to $J\cap \bigcup_{I\in e_1\cap e_0} I$, namely,
\begin{equation}
\label{eq:J_e_1_J_0_partition}
\mathcal{J}^{e_1}(J_0)=\bigsqcup_{J'\in\mathcal{J}^{e_1\cap e_0}:J'\supseteq J_0} \left\{J\in \mathcal{J}^{e_1}(J_0)\middle|J\cap \bigcup_{I\in e_1\cap e_0} I=J'\right\}.
\end{equation}
It is easily seen that the subset of the partition can be simplified as
\begin{equation}
\left\{J\in \mathcal{J}^{e_1}(J_0)\middle|J\cap \bigcup_{I\in e_1\cap e_0} I=J'\right\}=\{J\in \mathcal{J}^{e_1}|J\supseteq J'\}
\end{equation}
for $J'\supseteq J_0$.

Thus we can partition the summation $\sum_{J\in \mathcal{J}^{e_1}:J\supseteq J_0}w_J$ according to Eq.~\eqref{eq:J_e_1_J_0_partition}, i.e.,
\begin{equation}
\label{eq:sum_J_e_1_sup_J_0}
\sum_{J\in \mathcal{J}^{e_1}: J\supseteq J_0} w_J = \sum_{J'\in \mathcal{J}^{e_1\cap e_0}:J'\supseteq J_0} \sum_{J\in \mathcal{J}^{e_1}:J\supseteq J'} w_J.
\end{equation}
Similarly, partitioning the summation $\sum_{J\in \mathcal{J}^{e_0}:J\supseteq J_0}w_J$ yields
\begin{equation}
\label{eq:sum_J_e_0_sup_J_0}
\sum_{J\in \mathcal{J}^{e_0}:J\supseteq J_0} w_J=\sum_{J'\in \mathcal{J}^{e_1\cap e_0}:J'\supseteq J_0} \sum_{J\in \mathcal{J}^{e_0}:J\supseteq J'} w_J.
\end{equation}
The RHS of Eqs.~\eqref{eq:sum_J_e_1_sup_J_0} and \eqref{eq:sum_J_e_0_sup_J_0} are equal by applying Eq.~\eqref{eq:inters_implicit_assump1} to the inner summation. Therefore,
\begin{equation}
\label{eq:sum_J_e1_J_e0}
\sum_{J\in \mathcal{J}^{e_1}: J\supseteq J_0} w_J =
\sum_{J\in \mathcal{J}^{e_0}:J\supseteq J_0} w_J
\end{equation}


Applying the induction hypothesis for $e_2$ and $e_0$ and deducing similarly, we also obtain
\begin{equation}
\label{eq:sum_J_e2_J_e0}
\sum_{J\in \mathcal{J}^{e_2}: J\supseteq J_0} w_J=\sum_{J\in \mathcal{J}^{e_0}:J\supseteq J_0} w_J.
\end{equation}
Combining Eqs.~\eqref{eq:sum_J_e1_J_e0} and \eqref{eq:sum_J_e2_J_e0}, it can be seen that Eq.~\eqref{eq:inters_implicit_equal} holds for $w\in\MC^H_T$ with $e\leftarrow e_1$, $e'\leftarrow e_2$. 
\end{proof}

Theorem~\ref{thm:alpha_acyclic} actually shows a stronger result: When $H$ is $\alpha$-acyclic, $\MC^H_T$ is equal to $\MC^H$ and provides an explicit half-space representation of the latter. 
We prove it by induction on $|E|$. To apply the induction hypothesis, we partition $H$ into two subhypergraphs $H_1$, $H_2$ according to a leaf node of $T$. Theorem~\ref{thm:decomp} allows us to achieve a half-space representation of $\MC^H$ from those of $\MC^{H_k},k=1,2$. And Lemma~\ref{lem:MC_H_cap_eq_T} is important for verifying the resulting representation is consistent with $\MC^H_T$.



\begin{proof}[Proof of Theorem~\ref{thm:alpha_acyclic}]
Without loss of generality, assume that $H$ does not have any isolated vertices. Otherwise, let $V_0$ be the set of all isolated vertices of $H$, and $H'=(V\setminus V_0,E)$ be the hypergraph obtained by removing $V_0$ from $H$. Then $H'$ has the same set of hyperedges as $H$ and is also $\alpha$-acyclic. Additionally, the join tree of $H$ is also the join tree of $H'$. Since any $I\in V_0$ is an isolated vertex, for $i_0\in I$, the $w_{i_0}$ component of $w\in \mathscr{S}^{H}$ does not appear in the production constraints $w_J=\prod_{i\in J} w_i$ defined in Eq.~\eqref{eq:def_SG_MCMP_set}. Furthermore, because the indices of the remaining components are in $\mathcal{J}^{H'}$, we have
\begin{equation}
\label{eq:S_H_isolated}
\mathscr{S}^{H}=\left\{w\in \{0,1\}^{\mathcal{J}^H}\middle|\proj_{H'} w\in \mathscr{S}^{H'};w(I)=1,\ \forall\, I\in V_0\right\}.
\end{equation}
Let $\mathcal{J}_0=L\left(\bigcup_{I\in V_0} I\right)=\bigcup_{I\in V_0} \mathcal{J}^{\{I\}}\subseteq \mathcal{J}^H$. 
Taking the convex hull of both sides of 
Eq.~\eqref{eq:S_H_isolated} yields
\begin{subequations}
\begin{align}
\MC^{H}=\bigg\{w\in \mathbb{R}^{\mathcal{J}^H}\bigg|&\proj_{H'} w\in \MC^{H'};\\
&\proj_{\mathcal{J}_0} w\in \conv\left\{u\in \{0,1\}^{\mathcal{J}_0}\middle|u(I)=1,\ \forall\, I\in V_0\right\}\bigg\}.\label{eq:u_J_0_sum_1}
\end{align}
\end{subequations}
It is evident that Eq.~\eqref{eq:u_J_0_sum_1} is equivalent to
\begin{equation}
\proj_{\mathcal{J}_0} w\geq 0\ \text{and}\ w(I)=1,\ \forall\, I\in V_0.
\end{equation}
These inequalities are already contained in Eq.~\eqref{eq:H_repres_sum_I}. Thus, if $\MC^{H'}=\MC^{H'}_T$, then $\MC^{H}=\MC^{H}_T$.

In the rest of the proof, it is assumed that $H$ has no isolated vertices. We prove the theorem by induction on $|E|$. First, when $H$ contains no hyperedges, $T$ is an empty graph, and the theorem holds. When $|E|=1$, since $H$ has no isolated vertices, we have $H=(e,\{e\})$. Let $\Delta^{\mathcal{J}^e}$ be the standard simplex defined as
\begin{equation}
\Delta^{\mathcal{J}^e}:=\left\{u\in \mathbb{R}^{\mathcal{J}^e}_{\geq 0}\middle|\sum_{J\in \mathcal{J}^e} u_J=1\right\},
\end{equation}
all vertices of which are $0$-$1$ vectors. Now we define the linear map 
$\mathcal{T}: \mathbb{R}^{\mathcal{J}^e}\rightarrow \mathbb{R}^{\mathcal{J}^H}$ that maps $u$ to $w$ 
as follows:
\begin{subequations}
\begin{alignat}{2}
w_J&=u_J,\quad & \forall\, J&\in \mathcal{J}^e,\\
 w_i&= \sum_{J\in \mathcal{J}^e: J\ni i}u_J,\ & \forall\, i&\in I, I\in e.
\end{alignat}
\end{subequations}
Consider the polytope $P^H=\mathcal{T}(\Delta^{\mathcal{J}^e})$ obtained by mapping $\Delta^{\mathcal{J}^e}$ via $\mathcal{T}$. It is easy to see that its  half-space representation is given by
\begin{equation}
\label{eq:H_represent_1_hyperedge}
P^H=\bigg\{w\in \mathbb{R}_{\geq 0}^{\mathcal{J}^{H}}\bigg|w_i=\sum_{J\in \mathcal{J}^e: J\ni i}w_J,\ \forall\, I\in e, i\in I;\sum_{J\in \mathcal{J}^e} w_J=1\bigg\}.
\end{equation}
Moreover, the vertex set of $P^H$ is obtained by mapping the vertex set of $\Delta^{\mathcal{J}^e}$ via $\mathcal{T}$. 
This directly implies that the vertex set of $P^H$ is exactly $\mathscr{S}^H$, hence $P^H=\MC^H$. Note that in this case, the join tree of $H$ consists of a single node and no edges, so Eq.~\eqref{eq:equal_intersect_tree} does not exist. It is easy to verify that Eqs.~\eqref{eq:H_represent_1_hyperedge} and \eqref{eq:MC_H_relax_join_tree} are equivalent half-space representations: $\MC^H=P^H=\MC^H_T$, and the theorem holds for $|E|=1$.

Now suppose $|E|>1$. Let $H$ be a $\alpha$-acyclic hypergraph, and $T=(E,\mathcal{E})$ be a join tree of $H$. Let $e_1$ be a leaf node of $T$, and $e_2$ be the unique neighbor of $e_1$ in $T$. Then the path from $e_1$ to any node $e$ other than $e_1$ in $T$ contains $e_2$. By the definition of a join tree, $e\cap e_1\subseteq e_2$, implying $e\cap e_1\subseteq e_1\cap e_2$. Let $V_2=\bigcup_{e\in E-e_1} e$, then
\begin{equation}
e_1\cap V_2 = \bigcup_{e\in E-e_1} e_1\cap e \subseteq e_1\cap e_2.
\end{equation}
On the other hand, since $e_2\subseteq V_2$, we have $e_1\cap e_2\subseteq e_1\cap V_2$, hence
\begin{equation}
\label{eq:intersect_e_1_V_2}
e_1\cap V_2 = e_1\cap e_2.
\end{equation}
Let $V_1=e_1$, $p=e_1\cap e_2$. From Eq.~\eqref{eq:intersect_e_1_V_2}, we have $p=V_1\cap V_2$. Since $H$ has no isolated vertices, we have 
\[
V_1\cup V_2=\bigcup_{e\in E} e=V.
\]
Let $H_1=(V_1,\{e_1\})$ and $H_2=(V_2,E-e_1)$, both of which are subhypergraphs of $H$,  satisfying that $V_1\cup V_2=V$, hence $H=H_1\cup H_2$. Since $H_1$ has only one hyperedge, it is $\alpha$-acyclic, and its join tree is $T_1=(\{e_1\},\varnothing)$. Removing the node $e_1$ from $T$ yields the tree $T_2:=(E-e_1,\mathcal{E}-e_1e_2)$, which is the join tree of $H_2=(V_2,E-e_1)$, so $H_2$ is also $\alpha$-acyclic. By the induction hypothesis applied to $H_1$ and $H_2$, we have $\MC^{H_k}=\MC^{H_k}_{T_k}$ for $k=1,2$, namely,
\begin{subequations}
\label{eq:H_repres_H_1_single_hyperedge}
\begin{align}
\MC^{H_1}
=\bigg\{w\in \mathbb{R}_{\geq 0}^{\mathcal{J}^{H_1}}\bigg|
&w(I)=1,\ \forall\, I\in V_1;\\
&w_i =\sum_{J\in \mathcal{J}^{e_1}:J\ni i}w_J, \ \forall\, I\in V_1, i\in I\bigg\},
\end{align}
\end{subequations}
and
\begin{subequations}
\label{eq:H_repres_H2}
\begin{alignat}{1}
\MC^{H_2}=\bigg\{&w\in \mathbb{R}_{\geq 0}^{\mathcal{J}^{H_2}}\bigg|\nonumber\\
&w(I)=1,\ \forall\, I\in V_2;\\
&w_i =\sum_{J\in \mathcal{J}^e: J\ni i}w_J,\ \forall\, e\in E-e_1, i\in I\in e;\\
&\sum_{J\in \mathcal{J}^e: J\supseteq J_0} w_J =\sum_{J\in \mathcal{J}^{e'}: J\supseteq J_0}w_J,\ \forall\, ee'\in \mathcal{E}-e_1e_2,|e\cap e'|>1,J_0\in \mathcal{J}^{e\cap e'}\bigg\}.
\end{alignat}
\end{subequations}

If we can apply Theorem~\ref{thm:decomp} or Example~\ref{eg:simple_decomp} to $H=H_1\cup H_2$, then we have $\MC^{H}=\overline{\MC}^{H_1}\cap \overline{\MC}^{H_2}$. However, the conditions, $p\in (L(V_1)\cup E(H_1))\cap (L(V_2)\cup E(H_2))$ or $V_1\cap V_2=\varnothing$, may not hold generally.  To this end, we distinguish several cases to construct hypergraphs to fulfill them. 

\begin{description}
\item[1. $|p|\leq 1$.] Now either $|p|=1$ or $|p|=0$ holds. In the first case, $p=L(V_1)\cap L(V_2)$, so the condition $p\in (L(V_1)\cup E(H_1))\cap (L(V_2)\cup E(H_2))$ holds, and we have $\MC^{H}=\overline{\MC}^{H_1}\cap \overline{\MC}^{H_2}$. In the latter case, $p=\varnothing$, and by Example~\ref{eg:simple_decomp}, we also have $\MC^{H}=\overline{\MC}^{H_1}\cap \overline{\MC}^{H_2}$. Combining the equalities from Eqs.~\eqref{eq:H_repres_H_1_single_hyperedge} and \eqref{eq:H_repres_H2} gives the half-space representation of $\MC^H$, which can be easily verified to be identical to Eq.~\eqref{eq:MC_H_relax_join_tree}. Thus, $\MC^H=\MC^H_T$, and the theorem holds.
\item[2. $|p|>1$.] We consider when $p\notin E$ or $p\in E$ separately. The latter case is further divided into $p\in E-e_1$ and $p=e_1$. 
Accordingly, there are three cases under investigation.

\textbf{(1) $p\notin E$. } In this case, let $H_1'=(V_1,\{e_1,p\})$, $H_2'=(V_2,(E-e_1)\cup\{p\})$, and $H'=(V,E\cup\{p\})$. Then $H'=H_1'\cup H_2'$ satisfies the conditions of Theorem~\ref{thm:decomp}. Therefore, $\MC^{H'}=\overline{\MC}^{H_1'}\cap \overline{\MC}^{H_2'}$.

Note that $p\subset e_1$. By Proposition~\ref{prop:add_hyperedge_isomorph}, the half-space representation of $\MC^{H_1'}$ can be obtained by adding equalities to $\MC^{H_1}$, namely
\begin{subequations}
\label{eq:H_repres_H_1_prime}
\begin{align}
\MC^{H_1'}=\bigg\{w\in \mathbb{R}^{\mathcal{J}^{H_1'}}_{\geq 0}\bigg|
&w(I)=1,\  \forall\, I\in V_1;\\
& w_i=\sum_{J\in \mathcal{J}^{e_1}: J\ni i}w_J,\ \forall\, I\in V_1,i\in I;\\
&w_{J_0}=\sum_{J\in\mathcal{J}^{e_1}:J\supseteq J_0} w_{J},\  \forall\, J_0\in \mathcal{J}^{p}\bigg\}. \label{eq:cons_H_1_prime_e_1}
\end{align}
\end{subequations}
Similarly, since $p\subset e_2\in E-e_1$, the half-space representation of $\MC^{H_2'}$ can also be obtained by adding equalities to $\MC^{H_2}$, resulting in
\begin{equation}
\label{eq:H_repres_H_2_prime}
\MC^{H_2'}=\left\{w\in \mathbb{R}^{\mathcal{J}^{H_2'}}\middle|\proj_{H_2} w\in \MC^{H_2};w_{J_0}=\sum_{J\in\mathcal{J}^{e_2}:J\supseteq J_0} w_{J},\ \forall\, J_0\in \mathcal{J}^p\right\}.
\end{equation}
Since $\MC^{H'}$ is decomposable into $\MC^{H_1'}$ and $\MC^{H_2'}$, the half-space representation of $\MC^{H'}$ can be defined by the union of inequalities in Eqs.~\eqref{eq:H_repres_H_1_prime} and \eqref{eq:H_repres_H_2_prime}. Combining Eq.~\eqref{eq:H_repres_H2}, we obtain
\begin{subequations}
\label{eq:H_repres_H_prime}
\begin{align}
\MC^{H'}=\bigg\{&w\in \mathbb{R}_{\geq 0}^{\mathcal{J}^{H'}}\bigg|\nonumber\\
& w(I)=1,\ \forall\, I\in V; \label{eq:cons_multiple_choice}\\
& w_i=\sum_{J\in \mathcal{J}^{e}: J\ni e}w_J,\ \forall\, e\in E, i\in I\in e;\label{eq:cons_vert} \\
& \sum_{J\in \mathcal{J}^e: J\supseteq J_0} w_J =\sum_{J\in \mathcal{J}^{e'}: J\supseteq J_0}w_J,\ \forall\, ee'\in \mathcal{E}-e_1e_2,|e\cap e'|>1,J_0\in \mathcal{J}^{e\cap e'}; \label{eq:cons_T_neibr}\\
& w_{J_0}=\sum_{J\in\mathcal{J}^{e_1}:J\supseteq J_0} w_{J}=\sum_{J\in\mathcal{J}^{e_2}:J\supseteq J_0} w_{J},\ \forall\, J_0\in \mathcal{J}^p\bigg\}.\label{eq:cons_intersect_p}
\end{align}
\end{subequations}
By Proposition~\ref{prop:proj_subhgraph_MC_H}, $\MC^H=\proj_{H}\MC^{H'}$. Hence the half-space representation of $\MC^H$ can be obtained by applying Fourier-Motzkin elimination to Eq.~\eqref{eq:H_repres_H_prime} and eliminating the components in $\mathcal{J}^{H'}\setminus \mathcal{J}^H=\mathcal{J}^{p}$. As these components do not appear in Eqs.~\eqref{eq:cons_multiple_choice}--\eqref{eq:cons_T_neibr}, it is easily seen that
\begin{subequations}
\begin{align}
\MC^H=\bigg\{w\in \mathbb{R}_{\geq 0}^{\mathcal{J}^{H}}\bigg|&w\ \text{satisfies Eqs.~\eqref{eq:cons_multiple_choice}--\eqref{eq:cons_T_neibr}};\\
&\sum_{J\in\mathcal{J}^{e_1}:J\supseteq J_0} w_{J}=\sum_{J\in\mathcal{J}^{e_2}:J\supseteq J_0} w_{J},\ \forall\, J_0\in \mathcal{J}^{p}\bigg\}.\label{eq:cons_intersect_p2}
\end{align}
\end{subequations}
Clearly, Eq.~\eqref{eq:cons_intersect_p2} is the same as the specific equality in Eq.~\eqref{eq:equal_intersect_tree} where $e\leftarrow e_1$ and $e'\leftarrow e_2$, so $\MC^H=\MC^H_T$.

\textbf{(2) $p\in E-e_1$.} In this case, $p$ is a hyperedge of $H$. Let $H_1'=(V_1,\{e_1,p\})$, then $H=H_1'\cup H_2$ satisfies the conditions of Theorem \ref{thm:decomp}, so $\MC^{H}=\overline{\MC}^{H_1'}\cap \overline{\MC}^{H_2}$. Combining Eqs.~\eqref{eq:H_repres_H_1_prime} and \eqref{eq:H_repres_H2}, we obtain
\begin{equation}
\MC^H=\bigg\{w\in \mathbb{R}_{\geq 0}^{\mathcal{J}^{H}}\bigg|w\ \text{satisfies Eqs.~\eqref{eq:cons_multiple_choice}--\eqref{eq:cons_T_neibr} and \eqref{eq:cons_H_1_prime_e_1}}\bigg\};\\
\end{equation}



Note that equalities of Eqs.~\eqref{eq:cons_multiple_choice}--\eqref{eq:cons_T_neibr} are all contained in Eq.~\eqref{eq:MC_H_relaxation_intersect}, while Eq.~\eqref{eq:cons_H_1_prime_e_1} is the special case of Eq.~\eqref{eq:equal_inters_main} where $e\leftarrow p$, $e'\leftarrow e_1$. Thus equalities of $\MC^H$ are contained in those of $\MC^H_\cap$, implying $\MC^H\supseteq \MC^H_\cap$. However, $\MC^H_\cap$ is actually a relaxation of $\MC^H$. Combining Lemma~\ref{lem:MC_H_cap_eq_T}, we have $\MC^H=\MC^H_\cap=\MC^H_T$.

\textbf{(3) $p=e_1$.} In this case, $p$ is also a hyperedge of $H$. Let $H_2'=(V_2,(E-e_1)\cup\{p\})$, then $H=H_1\cup H_2'$ satisfies the conditions of Theorem \ref{thm:decomp}, so $\MC^{H}=\overline{\MC}^{H_1}\cap \overline{\MC}^{H_2'}$. Combining Eqs.~\eqref{eq:H_repres_H_1_single_hyperedge} and \eqref{eq:H_repres_H_2_prime}, we obtain
\begin{subequations}
\begin{align}
\MC^H=\bigg\{w\in \mathbb{R}_{\geq 0}^{\mathcal{J}^{H}}\bigg|&w\ \text{satisfies Eqs.~\eqref{eq:cons_multiple_choice}--\eqref{eq:cons_T_neibr}};\\
&w_{J_0}=\sum_{J\in\mathcal{J}^{e_2}:J\supseteq J_0} w_{J},\ \forall\, J_0\in \mathcal{J}^{p}\bigg\}.\label{eq:cons_intersect_p1}
\end{align}
\end{subequations}
Similar to case (2), it also holds that $\MC^H=\MC^H_T$.
\end{description}
\end{proof}

Since Problem~\eqref{eq:Linearized_MCPP} is equivalent to $\max_{w\in \MC^H} a\cdot w$ and linear programming admits polynomial-time algorithms \cite{grotschel_ellipsoid_1981}, 
Theorem~\ref{thm:alpha_acyclic} allows us to 
evaluate the complexity of MCPP when $H$ is $\alpha$-acyclic. 
We only need to estimate the dimension of $w$ 
as well as the number of inequalities and equalities in the half-space representation of $\MC^H$.   

We define the \emph{rank} of hypergraph $H$, denoted by $r$, as the maximum cardinality of its hyperedges,
and let 
\begin{equation*}
M = \max_{e\in E} |\mathcal{J}^{e}|,\ \ \text{and}\ \ M_\cap=\max_{e,e'\in E: e\cap e'\neq \varnothing} |\mathcal{J}^{e\cap e'}|.
\end{equation*}
With $\hat{d}=\max_{I\in V} |I|$, we can bound $M$ and $M_\cap$ as $M_\cap\leq M\leq \hat{d}^r$. It is easy to see that the dimension of $w$, $|\mathcal{J}^H|$, does not exceed $M(|V|+|E|)$. Noting that $|\mathcal{E}|=|E|-1$, it is readily seen that the number of equalities in the half-space representation of $\MC^H$ \eqref{eq:MC_H_repres_alpha_acyclic} does not exceed $|V|+|E|r\hat{d}+|E|M_\cap$, and the number of inequalities is $|\mathcal{J}^H|$. 
Thus, the number of inequalities and equalities is $\mathcal{O}(|V|+|E|)$ 
providing a constant upper bound on $r$ and $\hat{d}$. Combining the fact that the number of nonzero coefficients of original form~\eqref{eq:MCPP_pseudo_Bool_def} is at least $(|V|+|E|)$, MCPP is solvable in polynomial time.

\begin{remark}
Note that MCPP~\eqref{eq:MCPP_pseudo_Bool_def} is generalization of PBO~\eqref{eq:PBO} by letting $|I|=2, I\in V$. Therefore, PBO~\eqref{eq:PBO} is also polynomial solvable when $H$ is $\alpha$-acyclic and its rank is bounded ($\hat{d}=2$ in this case, which is already bounded).
We should mention that this observation does not contradict the NP-hardness of PBO~\eqref{eq:PBO} (also when $H$ is $\alpha$-acyclic) as shown in \cite{del_pia_complexity_2023}. Since in that work, the PBO instance reduced from MAX-CUT is associated with a hypergraph $H$ with rank equal to $|V|$, thus unbounded.
\end{remark}

\section{Proof of Theorem~\ref{thm:lifted_MP0}}
\label{sec:lifted_MP_H}

%

In order to prove Theorem~\ref{thm:lifted_MP0}, we need a lemma on $\MP^H$. Let $F\subseteq \MP^H$ be a facet of $\MP^H$ and $U\subseteq V$,
and define
\begin{align}
V_1^F &= \{I\in V|\forall\, z\in \MP^H, z_{I}=1\ \text{implies}\ z\in F\}, \\
E_U &= \{e\in L(V)\cup E\cup \{\varnothing\}|e\subseteq V\setminus U, e\cup U\in L(V)\cup E\}.
\end{align}

\begin{lemma}
\label{lem:MP_H_proj_full_dim}
If hypergraph $H$ is downward-closed and $U\subseteq V\setminus V_1^F$ is nonempty, then either $E_U$ is empty, or the projection of the face $F_U:=\{z\in F|z_I=0,\forall\, I\in U\}$ onto $\mathbb{R}^{E_U\setminus\{\varnothing\}}$, i.e., $\proj_{E_U\setminus\{\varnothing\}}F_U$ is full-dimensional.
\end{lemma}

\begin{proof}
Suppose $E_U$ is nonempty and the $\proj_{E_U\setminus\{\varnothing\}}F_U$ is not full-dimensional. It can be easily seen from the downward-closedness of $E$ that $\varnothing \in E_U$. 

Since $\proj_{E_U-\varnothing}F_U$ is not full-dimensional, there exists a nonzero vector $\hat{c}\in\mathbb{R}^{E_U-\varnothing}$ and $\delta'\in\mathbb{R}$ such that
\[
\hat{c}\cdot \proj_{E_U-\varnothing} z=\delta'
\]
for all $z\in F_U$. Padding $\hat{c}$ with zeros, we obtain $c'\in\mathbb{R}^{L(V)\cup E}$. Then we have for any $z\in F_U$,
\begin{equation}
c'\cdot z=\hat{c}\cdot\proj_{E_U-\varnothing} z=\delta'.
\end{equation}

For convenience, define $c'_{\varnothing}=-\delta'$ and $\prod_{I\in \varnothing} z_I=1$. 
We claim that the equality on $\mathbb{R}^{L(V)\cup E}$
\begin{equation}
\label{eq:conflict_F}
c\cdot z-\delta:=\mathscr{L}\left(\sum_{e\in E_U} c'_e\prod_{I\in e} z_I\right)\prod_{I\in U}(1-z_I)=0
\end{equation}
holds for all $z\in F$.
First, we have to certify that $\mathscr{L}$ is well-defined for $z\in \mathbb{R}^{L(V)\cup E}$. 
Expanding $\prod_{I\in U}(1-z_I)$ yields monomials of the form $\prod_{I\in e'} z_I$, where $e'\subseteq U$. Therefore, the linearization operator $\mathscr{L}$ requires components $z_{e\cup e'}$, where $e\in E_U$ and $e'\subseteq U$, with $z_\varnothing=1$. It follows from the definition of $E_U$ that
\[
e\cup e'\subseteq e\cup U\in L(V)\cup E.
\]
Combining the downward-closedness of $H$, we arrive at
\[
e\cup e'\in L(V)\cup E\cup \{\varnothing\}.
\]
Since 
\[
\sum_{e\in E_U} c'_e\prod_{I\in e} z_I=c'\cdot z-\delta'
\]
for $z\in \mathcal{S}^H$, 
now it suffices to verify that 
\begin{equation}
\label{eq:before_linearize}
(c'\cdot z-\delta')\prod_{I\in U}(1-z_I)=0
\end{equation}
holds for all vertices $z$ of $F$. Let us suppose $z\in \mathcal{S}^H\cap F$. (i) If there exists $I_1\in U$ such that $z_{I_1}=1$, then $1-z_{I_1}=0$, and the conclusion holds. (ii) If $z_I=0$ for all $I\in U$, then $z\in F_U$. By the assumption of $F_U$, we have $c'\cdot z-\delta'=0$, which means that Eq.~\eqref{eq:before_linearize} also holds.

The above dedution shows that $F\subseteq \{z\in \MP^H|c\cdot z=\delta\}$. Since $\MP^H$ is full-dimensional with $F$ being its facet, we have  $F=\{z\in \MP^H|c\cdot z=\delta\}$. However, from the definition of \eqref{eq:conflict_F}, for any $I\in U$ and $z\in \MP^H$, as long as $z_I=1$, we have $c\cdot z=\delta$, implying $z\in F$, and thus $U\subseteq V_1^F$, a contradiction.
\end{proof}

Theorem~\ref{thm:lifted_MP0} shows that Condition~\eqref{eq:lifted_facet_necc_suff} is necessary and sufficient for Eq.~\eqref{eq:lifted_valid_ineq0} to be facet-inducing. 
When the condition does not hold for $\bm S$, we construct $\bm S'$ and $\bm S''$ such that
\[
a(\bm S)\cdot w-\delta=(a(\bm S')\cdot w-\delta)+(a(\bm S'')\cdot w-\delta).
\]
Then $a(\bm S)\cdot w\leq \delta$ is sum of two valid inequalities and thus not facet-inducing. This is the basic idea for proving the necessity. As for the sufficiency, we prove it in the context of $\MC^H_\leq$ which is guaranteed by Corollary~\ref{corol:facet_MC_H_leq}. Suppose Condition~\eqref{eq:lifted_facet_necc_suff} holds and the face induced by
\begin{equation}
\label{eq:proj_a_v_delta}
\proj_\leq a(\bm S)\cdot v\leq \delta
\end{equation}
is contained in the facet induced by
\begin{equation}
\label{eq:a'_v_delta}
a'\cdot v\leq \delta'.
\end{equation}
We construct several classes of vertices of $\MC^H_\leq$ such that equality of Eq.~\eqref{eq:proj_a_v_delta} holds for them. 
By the inclusion relation, we can obtain properties of $a'$ and $\delta'$.  With the help of Lemma~\ref{lem:MP_H_proj_full_dim}, we are able to verify that Eq.~\eqref{eq:a'_v_delta} is a scalar multiplication of Eq.~\eqref{eq:lifted_valid_ineq0}, and hence they induce the same facet.

\begin{proof}[Proof of Theorem~\ref{thm:lifted_MP0}]
For convenience, let $c_{\varnothing}=-\delta$. Define the support of $c$ as
\begin{equation}
\supp c:=\{e\in L(V)\cup E\cup \{\varnothing\}|c_e\neq 0\}.
\end{equation}
It is clear that when $I\in \bigcap_{e\in \supp c} e$, if $z\in \MP^H$ satisfies $z_{I}=0$, then $z_e=0$ for all $e\in \supp c$ (note that $z_e\leq z_{I}$ is valid for $\MP^H$), thus
\begin{equation}
\sum_{e\in  L(V)\cup E\cup \{\varnothing\}} c_ez_e=\sum_{e\in \supp c} c_ez_e=0.
\end{equation}
Hence $\bigcap_{e\in \supp c} e\subseteq V_0$.

Similarly, if $I'\notin \bigcap_{e\in \supp c} e$, then there exists $e_0\in \supp c$ such that $I'\notin e_0$. We take $e_0$ to be the inclusionwise minimal one, then for any $e\in L(V)\cup E\cup \{\varnothing\}$ such that $e\subset e_0$, we have $c_e=0$ since $I'\notin e$. (Note that $e_0$ is possibly $\varnothing$.) We define $\hat{z}\in \mathcal{S}^H$ as follows. 
For each $e\in L(V)\cup E$, let $\hat{z}_e=1$ if $e\subseteq e_0$, and $\hat{z}_e=0$ otherwise. 
Now that $\hat{z}$ satisfies $\hat{z}_{I'}=0$, and
\begin{equation}
\sum_{e\in  L(V)\cup E\cup \{\varnothing\}} c_e \hat{z}_e=\sum_{e\in L(V)\cup E\cup \{\varnothing\}:e\subseteq e_0} c_e = c_{e_0}\neq 0,
\end{equation}
we obtain $I'\notin V_0$. By the arbitrariness of $I'$, it follows that $\bigcap_{e\in \supp c} e\supseteq V_0$.  
Combining $\bigcap_{e\in \supp c} e\subseteq V_0$, we have $\bigcap_{e\in \supp c} e=V_0$.

Let $\mathcal{P}^0$ denote the family of all non-empty proper subsets. We define 
\begin{equation}
a:\bigtimes_{I\in V} \mathcal{P}^0(I) \rightarrow \mathbb{R}^{\mathcal{J}^H}
\end{equation}
to represent the coefficients of $w$ on the LHS of Eq.~\eqref{eq:lifted_valid_ineq0} with argument being $(S_I)_{I\in V}$, i.e.,
\begin{equation}
a(\bm S)\cdot w=\sum_{e\in L(V)\cup E} c_e \mathscr{L} \prod_{I\in e} \sum_{i\in S_I} w_i,
\end{equation}
with $\bm S=(S_I)_{I\in V}\in \bigtimes_{I\in V} \mathcal{P}^0(I)$. To be convenient, let $a_{\varnothing}(\bm S)=-\delta$ (although we do not regard it as a component of $a(\bm S)$). It is clear that for $J\in \mathcal{J}^H$ such that $J\subseteq \bigcup_{I\in V} S_I$, $a_J(\bm S)=c_{E(J)}$, and for other $J\in \mathcal{J}^H$, $a_J(\bm S)=0$. Note that 
as long as $c_{E(J)}=0$, 
we have $a_J(\bm S)=0$. We write $\bigcup S_I:=\bigcup_{I\in V} S_I$ in short.
\\

\textit{Necessity.} We show necessity by proving that if there exists $I_0\in V_0$ such that $|S_{I_0}|>1$, or $I_1\in V_1$ such that $|S_{I_1}|<|I_1|-1$, then $a(\bm S)\cdot w\leq \delta$ is not facet-inducing for $\MC^H$. We distinguish two cases.

\textit{Case 1: there exists $I_0\in V_0$ such that $|S_{I_0}|>1$. }Let $S_{I_0}=S'_{I_0}\sqcup S''_{I_0}$ be a disjoint union, where $S'_{I_0}$ and $S''_{I_0}$ are nonempty. Then define
\begin{equation}
S'_I=\begin{cases}
S_I, & I\neq I_0\\
S'_{I_0}, & I=I_0
\end{cases},
\ \text{and}\ 
S''_I=\begin{cases}
S_I, & I\neq I_0\\
S''_{I_0}, & I=I_0
\end{cases}.
\end{equation}

We claim that
\begin{equation}
a(\bm S)=a(\bm S')+a(\bm S'').
\end{equation}
To prove the assertion, we verify the equality for its all components $J\in \mathcal{J}^H$. 
(i) If $a_J(\bm S)\neq 0$, then $J\subseteq \bigcup S_I$ and $E(J)\in \supp c$.
Since 
$V_0=\bigcap_{e\in \supp c} e\subseteq E(J)$ and $I_0\in V_0$, 
$J\cap I_0$ is a singleton. Let $J\cap I_0=\{i_0\}$, then we have $i_0\in S_{I_0}$. Noting that $S_{I_0}=S'_{I_0}\sqcup S''_{I_0}$ is a disjoint union, either $i_0\in S'_{I_0}$ or $i_0\in S''_{I_0}$ holds. 
When $i_0\in S'_{I_0}$, we have $J\subseteq \bigcup S'_I$ and $J\nsubseteq \bigcup S''_I$, thus $a_J(\bm S')=c_{E(J)}$ and $a_J(\bm S'')=0$; similarly when $i_0\in S''_{I_0}$, we have $a_J(\bm S')=0$ and $a_J(\bm S'')=c_{E(J)}$. 
In both cases, $a_J(\bm S)=a_J(\bm S')+a_J(\bm S'')$. (ii) If $a_J(\bm S)=0$, then at least one of $J\nsubseteq \bigcup S_I$ and $c_{E(J)}=0$ holds, it is easy to verify that $a_J(\bm S')=a_J(\bm S'')=0$, thus $a_J(\bm S)=a_J(\bm S')+a_J(\bm S'')$. Therefore, the assertion holds.

Since  $V_0=\bigcap_{e\in \supp c} e$ is nonempty, we have $\varnothing\notin \supp c$, implying $\delta=-a_{\varnothing}=0$. Thus
\begin{equation}
a(\bm S)\cdot w-\delta=(a(\bm S')\cdot w-\delta)+(a(\bm S'')\cdot w-\delta).
\end{equation}
By the definition of $a(\cdot)$, $a(\bm S')\cdot w\leq\delta$ is the lifted inequality of $c\cdot z\leq \delta$ w.r.t. $\bm S'$, thus also valid and not an implicit equality. The same property also holds for $a(\bm S'')\cdot w\leq\delta$. Moreover, considering the nonzero components of $a(\bm S)$ and $a(\bm S')$, it can be easily verified that $a(\bm S)\cdot w\leq\delta$ is not a scalar multiplication of $a(\bm S')\cdot w\leq\delta$. Therefore, as the sum of two valid inequalities, $a(\bm S)\cdot w\leq\delta$ is not facet-inducing.

\textit{Case 2: there exists $I_1\in V_1$ such that $|S_{I_1}|<|I_1|-1$. }Define $\bm S'$ by $S_{I_1}'=I_1\setminus S_{I_1}$ and  $S_I'=S_I$ for $I\neq I_1$. Thus, $|S_{I_1}'|>1$. Since for all 
$w\in \mathscr{S}^H$, $w$ satisfies the multiple choice constraints
\[
\sum_{i\in S_{I_1}}w_i+\sum_{i\in S_{I_1}'} w_i=1,
\]
we can replace $\sum_{i\in S_{I_1}}w_i$ with $1-\sum_{i\in S_{I_1}'} w_i$ in Eq.~\eqref{eq:lifted_valid_ineq0} before applying $\mathscr{L}$, then
\begin{subequations}
\label{eq:{eq:flipped_Ipp_all}}
\begin{align}
&\sum_{e\in L(V)\cup E} c_e \mathscr{L} \prod_{I\in e} \sum_{i\in S_I} w_i-\delta\\
=&-\delta+\sum_{e\in L(V)\cup E: e\not\ni I_1} c_e \mathscr{L} \prod_{I\in e} \sum_{i\in S_I'} w_i\\
&+\sum_{e\in L(V)\cup E: e\ni I_1} c_e \mathscr{L}\left(1-\sum_{i\in S_{I_1}'} w_i\right) \prod_{I\in e-I_1} \sum_{i\in S_I'} w_i. \label{eq:flipped_w_S_J1}
\end{align}
\end{subequations}
The possibility of linearization operator of the part Eq.~\eqref{eq:flipped_w_S_J1} being performed within the indices in $\mathcal{J}^H$ is guaranteed by the downward-closedness of $H$.

Similar to the $\psi_U$ defined in Eqs.~(3) and (4) of \cite{del_pia_polyhedral_2017} in the case of $|U|=1$, we define the affine mapping $\psi^{I}:\mathbb{R}^{L(V)\cup E}\rightarrow\mathbb{R}^{L(V)\cup E}$ for each $I\in V$ flipping the $I$th entry as follows:
\begin{subequations}
\begin{alignat}{1}
\psi^{I}_{e}(z)&=\mathscr{L}(1-z_{I})\prod_{\bar{I}\in e-I} z_{\bar{I}}, \ I\in e;\\
\psi^{I}_{e}(z)&=z_e, \ I\notin e,
\end{alignat}
\end{subequations}
where $e\in L(V)\cup E$. It is clear that $\psi^{I}$ is a bijection on $\mathcal{S}^H$. Now let 
\begin{equation}
\label{eq:reversed_c_prime_delta_prime}
c'\cdot z-\delta'=c\cdot \psi^{I_1}(z)-\delta.
\end{equation}
Then $c'\cdot z\leq \delta'$ is also facet-inducing.
Expressing Eq.~\eqref{eq:reversed_c_prime_delta_prime} more explicitly, we have
\begin{subequations}
\begin{align}
c'\cdot z-\delta'=\sum_{e\in L(V)\cup E: e\not \ni I_1} c_e z_e+\sum_{e\in L(V)\cup E: e \ni I_1} c_e \mathscr{L}(1-z_{I_1}) \prod_{I\in e-I_1} z_{I}-\delta.
\end{align}
\end{subequations}
Lifting the RHS of the above equation w.r.t. $\bm S'$ yields exactly the RHS of Eq.~\eqref{eq:{eq:flipped_Ipp_all}}, so does the LHS. 
Therefore, 
\begin{equation}
\sum_{e\in L(V)\cup E} c'_e \mathscr{L} \prod_{I\in e} \sum_{i\in S_I'} w_i-\delta'=
\sum_{e\in L(V)\cup E} c_e \mathscr{L} \prod_{I\in e} \sum_{i\in S_I} w_i-\delta
\end{equation}
holds for all $w\in \mathscr{S}^H$, thus an implicit equality for $\MC^H$. %
From the definition of $V_1$ in Eq.~\eqref{eq:def_V_0_V_11}, it follows that for any $z\in \MP^H$, if $z_{I_1}=0$, then $\psi^{I_1}_{I_1}(z)=1$, hence $c'\cdot z-\delta'=c\cdot \psi^{I_1}(z)-\delta=0$. Therefore, $I_1$ belongs to
\begin{equation}
V_0':=\{I\in V|\forall\, z\in \MP^H, z_{I}=0 \ \text{implies}\ c'\cdot z=\delta'\}.
\end{equation}
It is clear that $V_0'$ w.r.t. $c'\cdot z \leq\delta'$ corresponds to $V_0$ in Eq.~\eqref{eq:def_V_0_V_1} w.r.t. $c\cdot z \leq\delta$. Combining $|S_{I_1}'|>1$ and the argument in Case 1, we conclude that
\[
\sum_{e\in L(V)\cup E} c'_e \mathscr{L} \prod_{I\in e} \sum_{i\in S_I'} w_i\leq \delta'
\]
is not facet-inducing, so $a(\bm S)\cdot w\leq \delta$ is not facet-inducing either.\\

\textit{Sufficiency. } Suppose that the given $\bm{S}$ satisfies Condition \eqref{eq:lifted_facet_necc_suff}. In this part of the proof, we do not consider other $a(\bm{S})$, so we simply write $a$ for $a(\bm{S})$.

Considering the flipping technique described earlier, we can sequentially flip the components in $V_0$ of $z$. For each step, suppose we flip the entry $I'$, then we replace $c\cdot z\leq \delta$ with $c\cdot \psi^{I'}(z)\leq \delta$, and replace $S_{I'}$ with $I'\setminus S_{I'}$ correspondingly. By the observations regarding $V_0$ and $V_0'$ we made before, $V_0$ and $V_1$ will change to $V_0-I'$ and $V_1\cup \{I'\}$ respectively. Therefore, after all the flipping, we will arrive at $V_0=\varnothing$. Consequently, we may assume that $V_0=\varnothing$ without loss of generality.


For each $I\in V$, let $\bar{i}_{I}\in I\setminus S_{I}$ be arbitrarily chosen index, and define $D=\{\bar{i}_{I}\}_{I\in V}$.  Then for all $J\in\mathcal{J}^H$, $a_J\neq 0$ implies $J\in \mathcal{J}^H_\leq$. We will show that
\begin{equation}
\label{eq:valid_ineq_MC_leq}
\proj_{\mathcal{J}^H_\leq}{a}\cdot v\leq \delta
\end{equation}
is facet-inducing for $\MC_\leq^H$ in Eq.~\eqref{eq:MCPP_MC_leq_def}. By Corollary~\ref{corol:facet_MC_H_leq}, $a\cdot w\leq \delta$ is also facet-inducing for $\MC^H$. 


Let $F$ be the face induced by \eqref{eq:valid_ineq_MC_leq}, which is 
 contained in the facet $F'$ induced by $a'\cdot v \leq \delta'$, where $a',v\in\mathbb{R}^{\mathcal{J}^H_\leq}$. We set $a'_{\varnothing}=-\delta'$. To establish the sufficiency, we only need to demonstrate that $F$ is identical to $F'$, and the main task is to prove that there exists $\lambda>0$ such that
\begin{equation}
a'_J=\lambda a_J,\ \forall\, J\in \mathcal{J}^H_{\leq}\cup\{\varnothing\}.
\end{equation}


We first prove the case where $J\subseteq \bigcup S_I$.
Arbitrarily choose $\tilde{i}_I\in S_I$, and define $R=\{\tilde{i}_I\}_{I\in V}$. It is clear that $R\in \mathcal{J}^V|_{\bigcup S_I}$. Define a mapping $\mathcal{T}^R: \mathcal{S}^H\rightarrow \MC^H_\leq$ given by $z\mapsto v$ as follows:
\begin{subequations}
\begin{alignat}{2}
v_{\tilde{i}_I}&=z_{I}, &\quad I&\in V;\\
v_i&=0,&\quad i&\in [n]\setminus R,
\end{alignat}
\end{subequations}
and $v_J=\prod_{i\in J} v_i$ for all $J\in \mathcal{J}^H_{\leq}$ where $|J|>1$ to ensure $v\in \MC^H_\leq$.

It is easily seen that for $J\in \mathcal{J}^H_{\leq}$, 
$v_J\neq 0$ 
implies $J\subseteq R$, i.e., $J\in \mathcal{J}^H|_R$. In this case, 
\begin{equation}
v_J=\prod_{\tilde{i}_I\in J} v_{\tilde{i}_I}=\prod_{I\in E(J)} z_{I}=z_{E(J)}.
\end{equation}
The relation between $J$ and $E(J)$ is indeed a one-to-one correspondence between $\mathcal{J}^H|_R$ and $L(V)\cup E$. We denote the inverse mapping that maps $e=E(J)$ to $J=J(e,R)=R\cap \bigcup_{I\in e} S_I$. Then,
\begin{equation}
\label{eq:a_prime_z_to_v}
\sum_{e\in L(V)\cup E} a_{J(e,R)}'z_e=\sum_{J\in \mathcal{J}^H|_R} a_{J}'v_J=a'\cdot v\leq \delta'.
\end{equation}
By the arbitrariness of $z\in \mathcal{S}^H$, it follows that
\begin{equation}
\label{eq:a_prime_valid_ineq_z}
\sum_{e\in L(V)\cup E} a_{J(e,R)}'z_e\leq \delta'
\end{equation}
is valid for $\MP^H$.

Since $\sum_{i\in S_I} v_i=v_{\tilde{i}_I}=z_I$, we have
\begin{align*}
\proj_{\mathcal{J}^H_\leq} a\cdot v
&=\sum_{e\in L(V)\cup E}c_e\mathscr{L}\prod_{I\in e}\sum_{i\in S_I} v_i\\
&=\sum_{e\in L(V)\cup E}c_e\prod_{I\in e}z_{I}\\
&=c\cdot z.
\end{align*}

We denote the vertices on $\MP^H$ that satisfy $c\cdot z=\delta$ as $z^1,z^2,\dots,z^s\in \mathcal{S}^H$.
Define $v^k=\mathcal{T}^R(z^k)$ for $k=1,2,\dots,s$. Then $\proj_{\mathcal{J}^H_\leq} a\cdot v^k=c\cdot z^k=\delta$, thus $v^k\in F$, and consequently, Eq.~\eqref{eq:a_prime_z_to_v} becomes equality for $v\leftarrow v^k$. Hence, $z^k$ also makes Eq.~\eqref{eq:a_prime_valid_ineq_z} become equality, which implies that the facet induced by $c\cdot z\leq \delta$ is contained in the face defined by Eq.~\eqref{eq:a_prime_valid_ineq_z}. 
That is, Eq.~\eqref{eq:a_prime_valid_ineq_z} is also facet-inducing. Since $\MP^H$ is full-dimensional, these two facet-inducing inequalities, Eq.~\eqref{eq:a_prime_valid_ineq_z} and $c\cdot z\leq \delta$, are positive scalar multiple of each other. Namely, there exists a scalar $\lambda_{R}>0$ such that
\begin{equation}
\label{eq:lambda_a_J}
a'_{J(e,R)}=\lambda_{R} c_e,\ \forall\, e\in L(V)\cup E\cup \{\varnothing\},
\end{equation}
where we define $J(\varnothing,R)=\varnothing$.
By the definition of $\supp c$, the above equation implies
\begin{equation}
\label{eq:lambda_R_a_prime_c}
\lambda_{R}=a'_{J(e,R)}/c_e,\ \forall\, e\in \supp c.
\end{equation}
This equality  holds for any $R\in \mathcal{J}^V|_{\bigcup S_I}$. Note that given one $J\in \mathcal{J}^H|_{\bigcup S_I}\cup\{\varnothing\}$ such that $E(J)\in \supp c$, for all $R\in \mathcal{J}^V|_{\bigcup S_I}$ such that $R\supseteq J$, we have $J(E(J),R)=J$, implying that the corresponding $\lambda_R$ are identical.
Since $\bigcap_{e\in \supp c} e=V_0=\varnothing$, we can prove the equality of all $\lambda_R$s by selecting  different $J$s.

We strictly prove as follows: Suppose to the contrary that there exist $R,R'\in \mathcal{J}^V|_{\bigcup S_I}$ such that $\lambda_R\neq \lambda_{R'}$, then let $R,R'$ be the pair that maximizes $|R\cap R'|$ among them. Let $i_{I'}\in R\setminus R'$ and $i_{I'}'\in R'$, where $i_{I'},i'_{I'}\in S_{I'}\subset I'$. Obviously, $i_{I'}\neq i_{I'}'$. Since $\bigcap_{e\in \supp c} e=\varnothing$, there exists $e_0\in \supp c$ such that $I'\notin e_0$. Let $\hat{R}=(R-i_{I'})\cup\{i'_{I'}\}$, then $|\hat{R}\cap R'|=|R\cap R'|+1$. It can be readily observed that
\begin{align*}
J(e_0,R)&=R\cap \bigcup_{I\in e_0} S_I\\
&=(R-i_{I'})\cap \bigcup_{I\in e_0} S_I\\
&=(\hat{R}-i'_{I'})\cap \bigcup_{I\in e_0} S_I\\
&=J(e_0,\hat{R}),
\end{align*}
so $\lambda_R=a'_{J(e,R)}/c_e=a'_{J(e,\hat{R})}/c_e=\lambda_{\hat{R}}$, thus $\lambda_{\hat R}\neq \lambda_{R'}$, contradicting the maximality of $|R\cap R'|$.

Hence, we can let $\lambda>0$ such that $\lambda_R=\lambda,\forall\, R\in \mathcal{J}^V|_{\bigcup S_I}$. Now Eq.~\eqref{eq:lambda_a_J} becomes
\begin{equation}
a'_{J(e,R)}=\lambda c_e,\ \forall\, e\in L(V)\cup E\cup \{\varnothing\}.
\end{equation}
As $e$ ranges over $L(V)\cup E\cup \{\varnothing\}$ and $R$ ranges over $\mathcal{J}^V|_{\bigcup S_I}$, $J(e,R)$ ranges over $\mathcal{J}^H|_{\bigcup S_I}\cup\{\varnothing\}$. Thus, for any $J\in \mathcal{J}^H|_{\bigcup S_I}\cup \{\varnothing\}$, we have
\begin{equation}
a'_J=\lambda c_{E(J)}=\lambda a_J.
\end{equation}


Next, we only need to prove the case where $J\in\mathcal{J}^H_\leq$ but $J\nsubseteq \bigcup S_I$, and we will show that $a'_J=0$. Combining with $a_J=0$, we also have $a'_J=\lambda a_J$. 

Previously, we defined $v^k\in F$, and here we are going to construct more vertices on $F$ to obtain more properties using $\proj_{\mathcal{J}^H_\leq} a\cdot v=\delta$ when $v$ is on $F$.

Recall that $z^1,z^2,\dots,z^s\in \mathcal{S}^H$ are vertices of the facet defined by $c\cdot z\leq\delta$. Given any $U\subseteq V\setminus V_1$, let $K_U=\{k\in[s]|z^k_{I}=0,\forall\, I\in U\}$ be the index of these vertices where the components corresponding to $I\in U$ are all zero. 
We have 
\[
F_U = \{z\in \MP^H|c\cdot z=\delta;z_{I}=0,\forall\, I\in U\} =\conv \{z^k|k\in K_U\}. 
\]
For any $\tilde{i}_I\in S_I,I\in V\setminus U$ and $\tilde{i}_I\in (I-\bar{i}_I)\setminus S_I, I\in U$, we can define two sets of $v^k, \bar{v}^k\in \MC^H, k\in K_U$ as follows,
\begin{subequations}
\label{eq:v_k_bar_defn}
\begin{alignat}{2}
\bar{v}^k_{\tilde{i}_I} &= 1, &\quad I &\in U,\\
\bar{v}^k_{\tilde{i}_I} &= z^k_{I}, & I &\in V\setminus U,
\end{alignat}
\end{subequations}
and for remaining $i\in [n]$, $\bar{v}^k_i=0$. The definition of $v^k,k\in K_U$ is similar to before. 
Let $R=\{\tilde{i}_I\}_{I\in V}$, and $v^k=\mathcal{T}^R(z^k)$, i.e.,
\begin{equation}
\label{eq:v_k_defn}
v^k_{\tilde{i}_I}=z^k_{I},\quad I\in V\setminus U,
\end{equation}
and for remaining $i\in [n]$, $v^k_i=0$. $\bar{v}^k_J$ and $v^k_J$ for $|J|>1$ are defined by production constraints.
The only difference between $\bar{v}^k$ and $v^k$ is that $\bar{v}^k_{\tilde{i}_I}=1$ while $v^k_{\tilde{i}_I}=z^k_I=0$, where $I\in U$, and that the components $\bar{v}^k_J$ and $v^k_J$ with $J$ containing $\tilde{i}_I$. Note that $z^k_{I}=0, \forall\, k\in K_U, I\in U$, so we still have
\begin{equation}
\sum_{i\in I} v^k_i=\sum_{i\in I} \bar{v}^k_i=z^k_{I},
\end{equation}
and hence
\begin{equation}
a\cdot \bar{v}^k=a\cdot v^k=c\cdot z^k=\delta,\ \forall\, k\in K_U.
\end{equation}
By the assumption on $F$ and $F'$, $v^k,\bar{v}^k$ are on the face $F\subseteq F'$, implying
\begin{equation}
a'\cdot (\bar{v}^k-v^k)=a'\cdot \bar{v}^k-a'\cdot v^k=0.
\end{equation}

Next, we will express $a'\cdot (\bar{v}^k-v^k)$ in detail. Similar to the previous definition, 
let $J(e,R):=R\cap \bigcup_{I\in e} S_I$.

It is easily seen that for all $J\in \mathcal{J}^H_\leq$, $v^k_J\leq \bar{v}^k_J$. Thus $\bar{v}^k_J-v^k_J\neq 0$ if and only if $\bar{v}^k_J=1$ and $v^k_J=0$. By definitions Eqs.~\eqref{eq:v_k_bar_defn} and \eqref{eq:v_k_defn}, this holds only if
\begin{equation}
\label{eq:J_cond_neq_0}
J\cap\bigcup_{I\in U} I \neq \varnothing\ \text{and} \ J\subseteq R.
\end{equation}
Suppose $J_0$ satisfies the above condition and $e_0=E(J_0)$, then $J_0=J(e_0,R)$. Partition $e_0$ into subsets of $V\setminus U$ and $U$ respectively, denoted by $e'$ and $U'$, and defined by $e'=e_0\setminus U$, $U'=e_0\cap U$. Since 
$J\cap\bigcup_{I\in U} I\neq \varnothing$, we have 
$U'\neq \varnothing$.


For any nonempty $\bar{U}\subseteq V\setminus V_1$,  let 
\begin{equation}
E_{\bar{U}}=\{e\in L(V)\cup E\cup \{\varnothing\}|e\subseteq V\setminus \bar{U}, e\cup \bar{U}\in L(V)\cup E\}.
\end{equation}
Since $e'\cup U'=e_0\in L(V)\cup E$, we have $e'\in E_{U'}$. Note that for any $I\in U$, $\bar{v}^k_{\tilde{i}_I}=1$, thus
\begin{align*}
\bar{v}^k_{J_0}&=\prod_{i\in J(e',R)} \bar{v}^k_i\prod_{i\in J(U',R)} \bar{v}^k_i\\
&=\prod_{i\in J(e',R)} \bar{v}^k_i\\
&= \bar{v}^k_{J(e',R)}\\
&=z^k_{e'},
\end{align*}
where we define $\bar{v}^k_\varnothing=1$ and $z^k_\varnothing=1$.

Now $a'\cdot (\bar{v}^k-v^k)$ can be expressed by summing all $a'_J\cdot (\bar{v}^k_J-v^k_J)$ such that $J$ satisfies Eq.~\eqref{eq:J_cond_neq_0}. Noting that in this case $v^k_J=0$ (but $\bar{v}^k_J=1$ may not hold, as Eq.~\eqref{eq:J_cond_neq_0} is a necessary condition for $v^k_J\neq \bar{v}^k_J$), we have
\begin{subequations}
\label{eq:deduct_zero}
\begin{align}
0 &=a'\cdot (\bar{v}^k-v^k) \\
&= \sum_{U'\subseteq U: U'\neq\varnothing} \sum_{e'\in E_{U'}: e'\subseteq V\setminus U} a'_{J(e'\cup U',R)} \bar{v}^k_{J(e'\cup U',R)}\\
&= \sum_{U'\subseteq U: U'\neq\varnothing} \sum_{e'\in E_{U'}: e'\subseteq V\setminus U} a'_{J(e'\cup U',R)} z^k_{e'}.\label{eq:sum_U_prime_a_J_z_e}
\end{align}
\end{subequations}
The above equation holds for $k\in K_U$.

Now, we can prove $a'_J=0$ by induction on $|J\setminus \bigcup S_I|$. Note that we assume $J\in \mathcal{J}^H_\leq$ and $J\nsubseteq \bigcup S_I$, so $|J\setminus \bigcup S_I|>0$.

When $|J\setminus \bigcup S_I|=1$, let $e_0=E(J)$ and define
$\tilde{i}_I=J\cap I$ for $I\in e_0$. 
For $I\notin e_0$, we arbitrarily choose $\tilde{i}_I\in S_I$. Let $R=\{\tilde{i}_I\}_{I\in V}$, then $R\supseteq J$. Suppose $J\setminus \bigcup S_I=\{\tilde{i}_{I_0}\}$ and let $U=\{I_0\}$. Note that $\tilde{i}_{I_0}\in (I_0-\bar{i}_{I_0})\setminus S_{I_0}$ and $S_I=I-\bar{i}_I,\forall\, I\in V_1$, so $U\subseteq V\setminus V_1$. We define $K_U$ and $\bar{v}^k$, $v^k$, $k\in K_U$, by Eqs.~\eqref{eq:v_k_bar_defn} and \eqref{eq:v_k_defn}. Since $|U|=1$, Eq.~\eqref{eq:deduct_zero} becomes
\begin{equation}
\label{eq:a_J_prime_zero}
\sum_{e'\in E_{U}: e'\subseteq V\setminus U} a'_{J(e'\cup U,R)} z^k_{e'}=0.
\end{equation}
Note that the condition $e'\subseteq V\setminus U$ for summation can be omitted since it is already defined in $E_U$.
As $e_0\supseteq U$ and $e_0\in L(V)\cup E$, we know that $E_U$ is nonempty. By Lemma~\ref{lem:MP_H_proj_full_dim},
the projection of $F_U=\conv\{z^k|k\in K_U\}$ onto $\mathbb{R}^{E_U-\varnothing}$ 
is full-dimensional. However, Eq.~\eqref{eq:a_J_prime_zero} implies that for all points $z'$ on $\proj_{E_U-\varnothing} F_U$, it holds that
\begin{equation}
\sum_{e'\in E_{U}-\varnothing} a'_{J(e'\cup U,R)} z'_{e'}=a'_{J(U,R)},
\end{equation}
hence $a'_{J(e'\cup U,R)}=0,\forall\, e'\in E_U$. In particular, when $e'=e_0\setminus U$, we have $J=J(e'\cup U,R)$, yielding $a'_J=0$.

Suppose $|J\setminus \bigcup S_I|>1$. 
Similar to the previous case, let $e_0=E(J)$, define
$\tilde{i}_I=J\cap I$ for $I\in e_0$, 
and arbitrarily choose $\tilde{i}_I\in S_I$ for $I\notin e_0$. Likewise, we let $R=\{\tilde{i}_I\}_{I\in V}$ and $U=\{I\in V|\tilde{i}_{I}\in J\setminus \bigcup S_I\}$. We also define $K_U$ and $\bar{v}^k$, $v^k$ for $k\in K_U$. Again, we have $U\subseteq V\setminus V_1$ and Eq.~\eqref{eq:deduct_zero} holds.

Considering the index $U'$ of the outer sum in Eq.~\eqref{eq:sum_U_prime_a_J_z_e} where $U'\subset U$, it is easy to see that $|J(e'\cup U',R)\setminus \bigcup S_I|=|J(U',R)|=|U'|<|U|$, so the induction hypothesis implies $a'_{J(e'\cup U',R)}=0$. These parts of the outer sum vanish and the remainder is summing over $U'=U$, hence
\begin{equation}
\sum_{e'\in E_{U}} a'_{J(e'\cup U,R)} z^k_{e'}=0,\ \forall\, k\in K_U.
\end{equation}
Applying Lemma~\ref{lem:MP_H_proj_full_dim}, we conclude that $a_{J(e'\cup U,R)}=0$ holds for any $e'\in E_U$. Since $J=J(e_0,R)$ and $e_0\supseteq U$, we have $a'_J=0$.

In conclusion, there exists $\lambda > 0$ such that for any $J \in \mathcal{J}^H_\leq \cup \{\varnothing\}$, we have $a'_J = \lambda a_J$. Namely,
\[
a'\cdot v-\delta'=\lambda(a\cdot v-\delta),
\]
which indicates that the face defined by $a \cdot v \leq \delta$ is the same as the facet induced by $a' \cdot v \leq \delta'$. The sufficiency holds.
\end{proof}

In fact, the above proof also shows that when Eq.~\eqref{eq:lifted_facet_necc_suff} holds, Eq.~\eqref{eq:lifted_valid_ineq0} is also facet-inducing for $\MC^H_\leq(D)$ where $D\in \mathcal{J}^V$ and $D\subseteq [n]\setminus\bigcup S_I$. Thus we have the following corollary.
\begin{corollary} 
\label{corol:lifted_MC_leq}
Suppose the hypergraph $H(V,E)$ is downward-closed and $D\in \mathcal{J}^V$. Let $c \cdot z \leq \delta$ be a facet-inducing inequality for $\MP^H$. $V_0$ and $V_1$ are defined by Eqs.~\eqref{eq:def_V_0_V_1} and \eqref{eq:def_V_0_V_11}, respectively. 
For any $S_I\subset I\setminus D, S_I\neq \varnothing, I\in V$, the inequality
\begin{equation*}
\sum_{e\in L(V)\cup E} c_e \mathscr{L} \prod_{I\in e} \sum_{i\in S_I} v_i\leq \delta
\end{equation*}
is facet-inducing for $\MC^H_\leq$ if and only if Condition \eqref{eq:lifted_facet_necc_suff} holds.
\end{corollary}

There exists a similar operation named ``copying'' proposed in \cite{barmann_bipartite_2023} that introduces classes of valid inequalities for $P(G,\mathcal{I})$ inherited from BQP. The work specified two classes among them in its Theorems~3.6 and 3.7. Since BQP with multiple choice constraints is a special case of $\MC^H_\leq$ (which is isomorphic to $\MC^H$ when $H$ is a graph), it can be quickly verified that these theorems are corollaries to Corollary~\ref{corol:lifted_MC_leq}.  

\begin{example}
\label{eg:lifted_barmann}
Let $G(X\cup Y,E)$ be a bipartite graph with $X$ and $Y$ being color classes of $G$. 
Sections~3.2.1 and 3.2.2 of \cite{barmann_bipartite_2023} discussed the so called cycle inequalities and $I_{mm22}$ Bell inequalities of BQP, respectively, and provided two classes of facet-inducing inequalities for $P(G,\mathcal{I})$ which is achiedved by applying 0-lifting and copy operation to them. (See Theorems~3.6 and 3.7 of the paper.) The condition for $G$ is subset-uniform. Let $\mathcal{G}=(\mathcal{I}\cup Y,\mathcal{E})$ achieved by merging the vertices in each subset of the partition $\mathcal{I}$ into a single vertex. Since BQP is a special case of
  multilinear polytope, we can express the BQP w.r.t. $\mathcal{G}$ by $\MP^{\mathcal{G}}$.
The cycle inequality w.r.t. cycle $\{I_1I_2,I_2I_3,\dots,I_mI_1\}\subseteq \mathcal{E}$ reads
\begin{equation}
\label{eq:cycle_ineq}
-z_{I_1I_m}+\sum_{p=1}^{m-1} z_{I_p I_{p+1}}-\sum_{p=2}^{m-1} z_{I_p}\leq 0.
\end{equation}
Since $\mathcal{G}$ is bipartite, $m$ is even and $m\geq 4$. We assume that $I_1,I_3,\dots,I_{m-1}\in \mathcal{I}$.
It can be easily verified that $V_0=V_1=\varnothing$ for Eq.~\eqref{eq:cycle_ineq}. 
Define $H(V,\hat{E})$ and $D$ by Eqs.~\eqref{eq:V_from_P_G_I}--\eqref{eq:D_from_P_G_I} in Example~\ref{eg:BQP_MC_H}. Then $P(G,\mathcal{I})$ is in one-to-one correspondence with $\MC^H_\leq$. 
By Corollary~\ref{corol:lifted_MC_leq}, for any nonempty $S_{I_p}\subset I_p, p=1,3,\dots,m-1$ along with $S_{I_p}=I_p,p=2,4,\dots,m$, the lifted inequality of Eq.~\eqref{eq:cycle_ineq} w.r.t. $S_{I_p}$ is facet-inducing, and it is the same to the cycle+copying inequalities,  Eq.~(9) of \cite{barmann_bipartite_2023}. Thus Theorem 3.6 of that paper holds. 
For the case of $I_{mm22}$ Bell inequalities, it is also clear that $V_0=V_1=\varnothing$, and Theorem 3.7 of \cite{barmann_bipartite_2023} holds. 



\end{example}

\section{Conclusion and discussion}
\label{sec:con}

We proposed the Boolean polynomial polytope with multiple choice constraints, $\MC^H$, to study the feasible region of linearized MCPP~\eqref{eq:Linearized_MCPP}. Two main theorems were proved. The first gives the explicit half-space representation of $\MC^H$ when the hypergraph $H$ is $\alpha$-acyclic. The second provides a necessary and sufficient condition for the inequalities lifted from the facet-inducing ones for the multilinear polytope $\MP^H$ to be still facet-inducing for $\MC^H$. However, compared with the rich polyhedral study for the unconstrained case, 
there is still very little research on $\MC^H$ at current stage and many possible directions are left for the future.
We list some of them below.
\begin{itemize}
\item It will be interesting to develop separation algorithms for the lifted facet-inducing inequalities.

\item It is mentioned that the description of $\aff\MC^H$ for general hypergraph $H$ is still unclear in Section~\ref{sec:affine_hull}. A possible way to determine the implicit equalities for $\aff\MC^H$ is adding hyperedges into $H$ to achieve a downward-closed $H'$ such that $\MC^H$ is isomorphic to $\MC^{H'}$. Then the known properties of $\aff\MC^{H'}$ may be useful for $\aff\MC^H$.

\item Since $\MC^H_\leq$ is actually the same to $\MP^H$ when $|I|=2$ for all $I\in V$ (see Example~\ref{eg:MC_H_leq_iso_MP_H}), it is plausible to think that the results of $\MC^H$, e.g., Theorems~\ref{thm:alpha_acyclic} and \ref{thm:decomp},  may be specialized for $\MP^H$ through the isomorphism between $\MC^H$ and $\MC^H_\leq$ (under a certain condition). 

\item There is a less well-known acyclicity degree called ``cycle-free" defined in \cite{brault-baron_hypergraph_2016}. The condition for this cycle-freedom is even weaker than that for the $\alpha$-acyclicity. It may be worth studying the properties of $\MC^H$ when $H$ is cycle-free.

\item For small hypergraph $H$, there are special facet-inducing inequalities for $\MC^H$ other than those lifted from $\MP^H$ found by symbolic computation. Similar inequalities for general $\MC^H$ still need to be identified and classified.
\end{itemize}

\noindent
{\bf Acknowledgements}. 
This work was funded by the National Key R \& D Program of China (No.~2022YFA1005102) and
the National Natural Science Foundation of China (Nos.~12325112, 12288101).


\addcontentsline{toc}{section}{\refname}
\begin{thebibliography}{10}

\bibitem{barahona_cut_1986}
{\sc F.~Barahona and A.~R. Mahjoub}, {\em On the cut polytope}, Math. Program.,
  36 (1986), pp.~157--173.

\bibitem{beeri_desirability_1983}
{\sc C.~Beeri, R.~Fagin, D.~Maier, and M.~Yannakakis}, {\em On the desirability
  of acyclic database schemes}, J. ACM, 30 (1983), pp.~479--513.

\bibitem{boros_cut-polytopes_1993}
{\sc E.~Boros and P.~L. Hammer}, {\em Cut-polytopes, boolean quadric polytopes
  and nonnegative quadratic pseudo-boolean functions}, Math. Oper. Res., 18
  (1993), pp.~245--253.

\bibitem{boros_pseudo-boolean_2002}
\leavevmode\vrule height 2pt depth -1.6pt width 23pt, {\em Pseudo-{Boolean}
  optimization}, Discrete Appl. Math., 123 (2002), pp.~155--225.

\bibitem{brault-baron_hypergraph_2016}
{\sc J.~Brault-Baron}, {\em Hypergraph acyclicity revisited}, ACM Comput.
  Surv., 49 (2016), pp.~54:1--54:26.

\bibitem{barmann_bipartite_2023}
{\sc A.~Bärmann, A.~Martin, and O.~Schneider}, {\em The bipartite boolean
  quadric polytope with multiple-choice constraints}, SIAM J. Optim., 33
  (2023), pp.~2909--2934.

\bibitem{chandrasekaran_hypergraph_2021}
{\sc K.~Chandrasekaran, C.~Xu, and X.~Yu}, {\em Hypergraph k-cut in randomized
  polynomial time}, Math. Program., 186 (2021), pp.~85--113.

\bibitem{chopra_partition_1993}
{\sc S.~Chopra and M.~R. Rao}, {\em The partition problem}, Math. Program., 59
  (1993), pp.~87--115.

\bibitem{de_simone_cut_1990}
{\sc C.~De~Simone}, {\em The cut polytope and the {Boolean} quadric polytope},
  Discrete Math., 79 (1990), pp.~71--75.

\bibitem{deineko_approximability_2008}
{\sc V.~Deineko, P.~Jonsson, M.~Klasson, and A.~Krokhin}, {\em The
  approximability of {MAX} {CSP} with fixed-value constraints}, J. ACM, 55
  (2008), pp.~1--37.

\bibitem{del_pia_complexity_2023}
{\sc A.~Del~Pia and S.~Di~Gregorio}, {\em On the complexity of binary
  polynomial optimization over acyclic hypergraphs}, Algorithmica, 85 (2023),
  pp.~2189--2213.

\bibitem{del_pia_polyhedral_2017}
{\sc A.~Del~Pia and A.~Khajavirad}, {\em A polyhedral study of binary
  polynomial programs}, Math. Oper. Res., 42 (2017), pp.~389--410.

\bibitem{del_pia_multilinear_2018}
\leavevmode\vrule height 2pt depth -1.6pt width 23pt, {\em The multilinear
  polytope for acyclic hypergraphs}, SIAM J. Optim., 28 (2018), pp.~1049--1076.

\bibitem{del_pia_decomposability_2018}
\leavevmode\vrule height 2pt depth -1.6pt width 23pt, {\em On decomposability
  of multilinear sets}, Math. Program., 170 (2018), pp.~387--415.

\bibitem{del_pia_polynomial-size_2023}
\leavevmode\vrule height 2pt depth -1.6pt width 23pt, {\em A polynomial-size
  extended formulation for the multilinear polytope of beta-acyclic
  hypergraphs}, Math. Program.,  (2023).
\newblock doi: 10.1007/s10107-023-02009-4.

\bibitem{deza_clique-web_1992}
{\sc M.~Deza, M.~Grötschel, and M.~Laurent}, {\em Clique-web facets for
  multicut polytopes}, Math. Oper. Res., 17 (1992), pp.~981--1000.

\bibitem{fagin_degrees_1983}
{\sc R.~Fagin}, {\em Degrees of acyclicity for hypergraphs and relational
  database schemes}, J. ACM, 30 (1983), pp.~514--550.

\bibitem{frieze_improved_1997}
{\sc A.~Frieze and M.~Jerrum}, {\em Improved approximation algorithms for {MAX}
  k-{CUT} and {MAX} {BISECTION}}, Algorithmica, 18 (1997), pp.~67--81.

\bibitem{grotschel_ellipsoid_1981}
{\sc M.~Grötschel, L.~Lovász, and A.~Schrijver}, {\em The ellipsoid method
  and its consequences in combinatorial optimization}, Combinatorica, 1 (1981),
  pp.~169--197.

\bibitem{nauss_0-1_1978}
{\sc R.~M. Nauss}, {\em The 0-1 knapsack problem with multiple choice
  constraints}, European J. Oper. Res., 2 (1978), pp.~125--131.

\bibitem{padberg_boolean_1989}
{\sc M.~Padberg}, {\em The boolean quadric polytope: {Some} characteristics,
  facets and relatives}, Math. Program., 45 (1989), pp.~139--172.

\bibitem{schrijver_combinatorial_2002}
{\sc A.~Schrijver}, {\em Combinatorial {Optimization}: {Polyhedra} and
  {Efficiency}}, Algorithms and {Combinatorics}, Springer Berlin, Heidelberg,
  2002.

\bibitem{shao_ode_2023}
{\sc S.~Shao and Y.~Wu}, {\em An {ODE} approach to multiple choice polynomial
  programming}, Oct. 2023.
\newblock arXiv:2212.06371.

\bibitem{sherali_simultaneous_1995}
{\sc H.~D. Sherali, Y.~Lee, and W.~P. Adams}, {\em A simultaneous lifting
  strategy for identifying new classes of facets for the {Boolean} quadric
  polytope}, Oper. Res. Lett., 17 (1995), pp.~19--26.

\bibitem{yajima_polyhedral_1998}
{\sc Y.~Yajima and T.~Fujie}, {\em A polyhedral approach for nonconvex
  quadratic programming problems with box constraints}, J. Global Optim., 13
  (1998), pp.~151--170.

\end{thebibliography}
\end{document}